\documentclass[a4paper,12pt]{article}
\usepackage{currfile,datetime}
\usepackage{amsthm,amsmath,amssymb,bbm,geometry,epsfig,listings,hyperref,paralist,
enumerate,comment,nicefrac,verbatim,multirow,xcolor}
\usepackage[capitalise,compress]{cleveref} % clever reference

\crefname{equation}{}{}

\newtheorem{lemma}{Lemma}[section]
\newtheorem{remark}[lemma]{Remark}

\newtheorem{theorem}[lemma]{Theorem}

\newtheorem{corollary}[lemma]{Corollary}
\newtheorem{setting}[lemma]{Setting}

\crefname{subsection}{Subsection}{Subsections}
\crefname{framework}{Framework}{Frameworks}
\crefname{setting}{Setting}{Settings}
\crefname{enumi}{Item}{Items}
\creflabelformat{enumi}{(#2\textup{#1}#3)}
\newcommand{\concat}{\odot}

\newcommand{\supnorm}[1]{{\left\vert\kern-0.25ex\left\vert\kern-0.25ex\left\vert #1 
    \right\vert\kern-0.25ex\right\vert\kern-0.25ex\right\vert}}
\newcommand{\cardna}{\dim}
\newcommand{\card}[1]{\dim\!\left( #1\right)}

\newcommand{\1}{\ensuremath{\mathbbm{1}}}
\providecommand{\N}{{\ensuremath{\mathbbm{N}}}}
\providecommand{\Z}{{\ensuremath{\mathbbm{Z}}}}
\providecommand{\R}{{\ensuremath{\mathbbm{R}}}}

\providecommand{\E}{{\ensuremath{\mathbb{E}}}}

\renewcommand{\P}{{\ensuremath{\mathbb{P}}}}

\newcommand{\funcPhi}[1]{\Phi_{#1}}

\newcommand{\netflow}[2]{\Phi_{#1}^{#2}}
\newcommand{\funcSmallF}{f}
\newcommand{\uniform}{\ensuremath{\mathcal{U}}}
\newcommand{\unif}{\ensuremath{\mathfrak{u}}}
\newcommand{\Exp}[1]{ \E \! \left[ #1 \right]}

\DeclareMathOperator*{\dimsum}{\boxplus}
\newcommand{\funcF}{F}% the function F
\newcommand{\funcG}{g}% the function G

\newcommand{\LipConstF}{L}% Lipschitz constant

\newcommand{\fwpr}{\ensuremath{\mathbf{W}}} %forward process
\newcommand{\sppr}{\ensuremath{W}}

\newcommand{\bigU}{U}%approximating solutions in Section 3
\newcommand{\normRd}[1]{\left\|{#1}\right\|}
\newcommand{\smallU}{u}% solution of the fixpoint equation...
\newcommand{\boundFG}{B}
\newcommand{\cardL}[1]{\dim\!\left(\funcCalL\left(#1\right)\right)}
\newcommand{\setCalN}{\mathbf{N}}% set of all neural networks
\newcommand{\funcCalP}{\mathcal{P}}% number of parameters

\newcommand{\funcCalR}{\mathcal{R}}
\newcommand{\setCalD}{\mathbf{D}}
\newcommand{\netapproxg}[2]{\mathfrak{g}_{#1,#2}}
\newcommand{\netapproxf}[1]{\mathfrak{f}_{#1}}
\newcommand{\identity}[1]{\mathrm{Id}_{#1}}
\newcommand{\funcCalL}{\mathcal{D}}

\newcommand{\neutralDim}[1]{\mathfrak{n}_{#1}}
\newcommand{\funcBoldA}[1]{\mathbf{A}_{#1}}

\newcommand{\eps}{\varepsilon}
\newcommand{\gp}{\xi}

\title{
A proof that rectified deep neural networks\\
overcome the curse of dimensionality in the numerical\\
approximation of semilinear
heat equations}

\author{
Martin Hutzenthaler$^{1}$, Arnulf Jentzen$^{2}$, Thomas Kruse$^{3}$, \& Tuan Anh Nguyen$^{4}$ 
\bigskip
\\
\small{$^1$ Faculty of Mathematics, University of Duisburg-Essen,}
\\
\small{45117 Essen, Germany, e-mail: martin.hutzenthaler@uni-due.de}
\smallskip
\\
\small{$^2$ SAM, Department of Mathematics,
ETH Zurich,}
\\
\small{8092 Zurich, Switzerland, e-mail: arnulf.jentzen@sam.math.ethz.ch}
\smallskip
\\
\small{$^3$ Institute of Mathematics, University of Gie{\ss}en,}
\\
\small{35392 Gie{\ss}en, Germany, e-mail: thomas.kruse@math.uni-giessen.de}
\smallskip
\\
\small{$^4$ Faculty of Mathematics, University of Duisburg-Essen,}
\\
\small{45117 Essen, Germany, e-mail: tuan.nguyen@uni-due.de}
}

\begin{document}

\maketitle
\makeatletter
\let\@makefnmark\relax
\let\@thefnmark\relax
\@footnotetext{\emph{AMS 2010 subject classification:} 65C99;
%65M99, 60H30,
68T05}
\@footnotetext{\emph{Key words and phrases:} curse of dimensionality, high-dimensional PDEs, deep neural networks, information based complexity,
tractability of multivariate problems, 
%high-dimensional semilinear BSDEs, 
multilevel Picard approximations
%, multilevel Monte Carlo method
}
\makeatother

\begin{abstract}
 Deep neural networks and other deep learning methods have very successfully
 been applied to the numerical approximation of high-dimensional
 nonlinear
 parabolic partial differential equations (PDEs), which are widely used in finance, engineering,
 and natural sciences.
 In particular, simulations indicate that algorithms based on deep learning
 overcome the curse of dimensionality in the numerical approximation
 of solutions of semilinear PDEs. For certain linear PDEs this has also been proved
 mathematically.
 The key contribution of this article is to rigorously prove this for the first time for a class of nonlinear PDEs.
 More precisely,
 we prove
 in the case of semilinear heat equations with gradient-independent nonlinearities
 that the numbers of parameters of the employed deep neural networks
 grow at most polynomially in both the PDE dimension and
 the reciprocal of the prescribed approximation accuracy.
 Our proof relies on recently introduced full history recursive multilevel Picard approximations
 of semilinear PDEs.
\end{abstract}

\tableofcontents

\section{Introduction}
Deep neural networks (DNNs) have revolutionized
a number of computational problems; see, e.g., the references in Grohs et al.\ \cite{GHJvW18}.
In 2017 deep learning-based approximation algorithms for certain parabolic partial differential equations (PDEs) have been proposed in Han~et~al.~\cite{EHJ17,HJE18} and based on these works there is now a series of deep learning-based numerical approximation algorithms for a large class of different kinds of PDEs in the scientific literature; see, e.g.,
\cite{BBG+18,BEJ17,BCJ18,EY18,EGJS18,FTT17,GHJvW18,Hen17,KLY17,Mis18,NM18,Rai18,SS17}.
%Recently,
%various deep learning methods have been proposed for numerical approximations of
%parabolic partial differential equations (PDEs);
%see, e.g., 
%
There is empirical evidence that deep learning-based methods work exceptionally well
for approximating solutions of high-dimensional PDEs and that these do not suffer from the
\emph{curse of dimensionality}; see, e.g., the simulations in
\cite{EHJ17,HJE18,BEJ17,BBG+18}.
There exist, however, only few theoretical results which prove
that DNN approximations of solutions of PDEs do not suffer
from the curse of dimensionality:
The recent articles \cite{GHJvW18,BGJ18,JSW18, EGJS18}
prove rigorously that DNN approximations overcome the curse 
of dimensionality in the numerical 
approximation of solutions of certain linear PDEs.

The main result of this article, Theorem~\ref{n18} below, proves 
for
semilinear heat equations
with gradient-independent nonlinearities
that the number of parameters
of the approximating DNN grows at most polynomially in both
the PDE dimension $d\in\N$ and the reciprocal of the prescribed accuracy $\eps>0$.
Thereby, we establish for the first time that there exist DNN approximations of solutions of such PDEs
which indeed overcome  the curse of dimensionality.
To illustrate the main result of this article we formulate in the following result,
Theorem~\ref{thm:intro} below, a special case of Theorem~\ref{n18}.
%To illustrate Theorem~\ref{n18}, we formulate
%the following special case of Theorem~\ref{n18}.
% using the above notation on DNNs.
%and the notation from Subsection~\ref{ssuc:Notation}.
%
%
%
%\sloppy
\begin{theorem}\label{thm:intro}
Let $\funcBoldA{d}\colon \R^d\to\R^d$, $d\in \N=\{1,2,\ldots\}$, and $\left\|\cdot \right\|\colon (\cup_{d\in \N} \R^d) \to [0,\infty)$
satisfy for all $d\in \N$, $x=(x_1,\ldots,x_d)\in \R^d$ that 
$
\funcBoldA{d}(x)= \left(\max\{x_1,0\},\ldots,\max\{x_d,0\}\right)
$
and 
$\|x\|=[\sum_{i=1}^d(x_i)^2]^{1/2}$,
let
$
\setCalN= \cup_{H\in  \N}\cup_{(k_0,k_1,\ldots,k_{H+1})\in \N^{H+2}}
[ \prod_{n=1}^{H+1} \left(\R^{k_{n}\times k_{n-1}} \times\R^{k_{n}}\right)],
$
let
$\funcCalR\colon \setCalN\to (\cup_{k,l\in \N} C(\R^k,\R^l))$ and
$\funcCalP\colon\setCalN\to \N$
 satisfy
for all $H\in \N$, $k_0,k_1,\ldots,k_H,k_{H+1}\in \N$,
$
\Phi = ((W_1,B_1),\ldots,(W_{H+1},B_{H+1}))\in \prod_{n=1}^{H+1} \left(\R^{k_n\times k_{n-1}} \times\R^{k_n}\right), 
$
$x_0 \in \R^{k_0},\ldots,x_{H}\in \R^{k_{H}}$ with 
$\forall\, n\in \N\cap [1,H]\colon x_n = \funcBoldA{k_n}(W_n x_{n-1}+B_n )$ 
that
\vspace{-5mm}
\begin{equation*}
\funcCalR(\Phi )\in C(\R^{k_0},\R ^ {k_{H+1}}),\quad
 (\funcCalR(\Phi)) (x_0) = W_{H+1}x_{H}+B_{H+1}, \quad\text{and}\quad \funcCalP(\Phi )=\textstyle{\sum\limits_{n=1}^{H+1}}k_n(k_{n-1}+1),
\end{equation*}
\vspace{-6mm}

\noindent
let
$T,\kappa \in (0,\infty)$, $f\in C(\R,\R)$, $(\netapproxg{d}{\varepsilon})_{d\in \N,\varepsilon\in (0,1]}\subseteq \setCalN$,
$(c_d)_{d\in\N}\subseteq (0,\infty)$,
for every $d\in \N$ let $g_d\in C(\R^d,\R)$, 
  for every $d\in \N$ let $u_d\in C^{1,2}([0,T]\times\R^d,\R)$,
%for every $\varepsilon\in (0,1]$, $d\in \N$ let $\netapproxg{d}{\varepsilon}\in \setCalN$,
and assume for all $d\in \N$, $v,w\in \R$, $x\in \R^d$, $\varepsilon \in (0,1]$, $t\in(0,T)$ that
$|f(v)-f(w)|\le \kappa|v-w|$,
$\funcCalR(\netapproxg{d}{\varepsilon})\in C(\R^d,\R)$,
$|(\funcCalR(\netapproxg{d}{\varepsilon}))(x)|\le \kappa d^\kappa (1+\normRd{x}^\kappa)$,
 $\left| g_d(x)-(\funcCalR(\netapproxg{d}{\varepsilon}))(x)\right| \le \varepsilon \kappa d^\kappa (1+\normRd{x}^\kappa)$,
 $\funcCalP(\netapproxg{d}{\varepsilon})\le \kappa d^\kappa \varepsilon^{-\kappa}$,
%$\left(\int_{\R^d}\normRd{y}^{2pq} \nu_d(dy)\right)^{\nicefrac{1}{(2pq)}}\le \boundFG d^{\mathfrak p}$,
 $|u_d(t,x)|\le c_d(1+\normRd{x}^{c_d})$,
$u_d(0,x)=g_d(x)$,
and
\begin{equation}  \begin{split}\label{eq:intro.PDE}
  (\tfrac{\partial}{\partial t}u_d)(t,x)=(\Delta_xu_d)(t,x)+f(u_d(t,x)).
\end{split}     \end{equation}
Then
there exist $(\Psi_{d,\varepsilon})_{d\in \N,\varepsilon \in (0,1]}\subseteq \setCalN$, $\eta \in (0,\infty)$
%, $C=(C_\gamma)_{\gamma \in (0,1]}\colon (0,1]\to (0,\infty)$ 
such that for all
$d\in \N$, $\varepsilon\in (0,1]$ it holds that
$\funcCalR(\Psi_{d,\varepsilon})\in C(\R^d,\R)$, 
$\funcCalP(\Psi_{d,\varepsilon})\le\eta  d^\eta \varepsilon^{-\eta}$, and
\begin{equation}
\left[\int_{[0,1]^d}\left|\smallU_d(T,x)-(\funcCalR(\Psi_{d,\varepsilon}))(x)\right|^2\, dx\right]^{\!\nicefrac{1}{2}}\le \varepsilon.
\end{equation}
\end{theorem}

\fussy
\cref{thm:intro} is an immediate consequence of \cref{cor:main_thm} in \cref{subsec:dnn_approx_gen_pol} below
(with $T=2T$, $u_d(t,x)=u_d(T-\frac{t}{2},x)$, $f(v)=f(v)/2$ for $t\in [0,2T]$, $x\in \R^d$, $v\in \R$ in the notation of 
\cref{cor:main_thm}). 
In the manner of the proof of Theorem 3.14 in \cite{GHJvW18} and the proof of Theorem 6.3 in \cite{JSW18}, the proof of \cref{n18} below uses probabilistic arguments on a suitable artificial probability space.
Moreover, the proof of \cref{n18} relies on recently introduced full history recursive multilevel Picard (MLP) approximations
which have been proved to overcome the curse of dimensionality
in the numerical approximation of solutions of semilinear heat equations at single space-time points;
see 
\cite{EHJK16,EHJK17,HK17,HJK+18}.
A key step in our proof is that realizations of certain MLP approximations
can be represented by DNNs; see \cref{b26} below.

The remainder of this article is organized as follows. In \cref{sec:stab_MLP} we provide auxiliary results on multilevel Picard approximations ensuring that these approximations are stable against perturbations in the nonlinearity $f$ and the terminal condition $g$ of the PDE \cref{eq:intro.PDE}.
In \cref{sec:Picard} we show that multilevel Picard approximations can be represented by DNNs and we provide bounds for the number of parameters of the representing DNN. 
We use the results of \cref{sec:stab_MLP} and \cref{sec:Picard} to prove the main result \cref{n18} in \cref{sec:main_result}.

\section[A stability result for multilevel Picard (MLP) approximations]{A stability result for full history recursive multilevel Picard (MLP) approximations}
\label{sec:stab_MLP}
\subsection{Setting}
\begin{setting}
\label{t00}
Let $d\in \N$, $T,L,\delta,\boundFG \in (0,\infty)$, $p,q\in [1,\infty)$, $f_1,f_2\in C\left( [0,T]\times \R^{d}\times\R,\R\right)$,
$g_1,g_2\in C(\R^d,\R)$,
% let $\normRd{\cdot}$ be the $d$-dimensional Euclidean norm,
let $\|\cdot \|\colon \R^d \to [0,\infty)$ satisfy for all $x=(x_1,\ldots,x_d)\in \R^d$ that $\|x\|=[\sum_{i=1}^d(x_i)^2]^{1/2}$,
assume for all $t\in [0,T]$,
$x\in \R^d$,
$w,v\in \R$, $i\in \{1,2\}$ that
\begin{align}\label{t01}
\left|f_i(t,x,w)-f_i(t,x,v)\right|\leq \LipConstF\left|w-v\right|,
\end{align}
\begin{align}
\max\left\{\big.\!\left|\funcSmallF_i(t,x,0)\right|,\left|\funcG_i(x)\right|\right\}\leq \boundFG\left(1+\normRd{x}\big.\!\right)^p,\label{m01}
\end{align}
and
\begin{align}
\max\left\{\big.\!\left|\funcSmallF_1(t,x,v)-\funcSmallF_2(t,x,v)\right|, \left|\funcG_1(x)-\funcG_2(x)\right|\right\}\leq \delta\left(\Big.\!\left(\big.1+\normRd{x}\right)^{pq}+|v|^q\big.\right),\label{m02}
\end{align}
let $\funcF_i\colon C\left([0,T]\times \R^d,\R\right)\to C\left([0,T]\times \R^d,\R\right)$, $i\in \{1,2\}$,
satisfy for all $v\in C\left([0,T]\times \R^d,\R\right)$, $t\in [0,T]$, 
$x\in  \R^d$, $i\in \{1,2\}$ that
\begin{equation}  \label{t02}
(\funcF_i(v))(t,x) = f_i(t,x,v(t,x)),
\end{equation}
let
$(\Omega, \mathcal{F}, \P)$
be a probability space, let $\fwpr \colon [0,T]\times \Omega\to \R^d$ be a standard 
%$d$-dimensional 
Brownian motion with continuous sample paths,
let $\smallU_1,\smallU_2\in C([0,T]\times \R^d,\R)$, assume  
for all $i\in \{1,2\}$, $s\in[0,T]$, $x\in\R^d$ that
\begin{align}
\E\!\left[\left|g_i\left(x+\fwpr_{T-s}\right)\big.\!\right|+\int_s^T\left| \left(\funcF_i(\smallU_i)\right)\left(t,x+\fwpr_{t-s}\right)\right|\,dt
\right]<\infty
\end{align}
and 
\begin{align}\label{a01}
\smallU_i(s,x)=\E\!\left[\funcG_i\left(x+\fwpr_{T-s}\right)+\int_s^T \left(\funcF_i(\smallU_i)\right)\left(t,x+\fwpr_{t-s}\right)\,dt\right],
\end{align} 
%let $\nu \colon \mathcal{B}(\R^d)\to[0,1]$ be a probability measure,
%for every $k \in \N_0$ and 
%every $(\mathcal{B}([0, T] \times \R^d) \otimes \mathcal{F}) / \mathcal{B}(\R)$-measurable function 
%$V \colon [0, T] \times \R^d \times \Omega \to \R$ let
%$
%\left\|V\right\|_{k} \in [0,\infty]
%$
%be the extended real number given by
%\begin{align}
%\begin{aligned}
%  \left\|V\right\|_{k}^2
%  =
%  \begin{cases}\displaystyle\int_{\R^d}
% \E\!\left[\left|V(0,x)\right|^2\vphantom{\big|}\right]\nu_d(dx)&\qquad 
%\colon  k= 0\\[10pt] \displaystyle
%   \frac{1}{T^{k}} \int_{\R^d}\int_0^{T}
%\frac{t^{k-1}}{(k-1)!}\,  \E\!\left[
%  \left|V(t,x+\fwpr_t)\right|^2\vphantom{\big|}\right]dt\,\nu_d(dx)&\qquad 
%\colon
%  k\geq 1.
%  \end{cases}
%\end{aligned}\label{a03}
%\end{align}
let
$  \Theta = \bigcup_{ n \in \N } \Z^n$,
let $\unif^\theta\colon \Omega\to[0,1]$, $\theta\in \Theta$, be independent random variables which are uniformly distributed on $[0,1]$, 
let $\uniform^\theta\colon [0,T]\times \Omega\to [0, T]$, $\theta\in\Theta$, satisfy 
for all $t\in [0,T]$, $\theta\in \Theta$ that 
$\uniform^\theta _t = t+ (T-t)\unif^\theta$,
let $\sppr^\theta\colon [0,T]\times\Omega\to\R^d $, $\theta\in \Theta$, be independent
% $d$-dimensional 
standard Brownian motions,
assume that $(\unif^\theta)_{\theta\in \Theta}$, $(\sppr^\theta)_{\theta\in \Theta}$, and $\fwpr$ are independent,
and
let
$ 
  {\bigU}_{ n,M}^{\theta } \colon [0, T] \times \R^d \times \Omega \to \R
$, $n,M\in\Z$, $\theta\in\Theta$, be functions
which satisfy
for all $n,M \in \N$, $\theta\in\Theta $, 
$ t \in [0,T]$, $x\in\R^d $
that ${\bigU}_{-1,M}^{\theta}(t,x)={\bigU}_{0,M}^{\theta}(t,x)=0$ and 
%\begin{equation}  \begin{split}
\begin{equation}
\begin{split}\label{t27}
  &{\bigU}_{n,M}^{\theta}(t,x)
=
  \frac{1}{M^n}
 \sum_{i=1}^{M^n} 
      \funcG_2\left(x+\sppr^{(\theta,0,-i)}_{T}-\sppr^{(\theta,0,-i)}_{t}\right)
 \\
  +&
  \sum_{l=0}^{n-1} \frac{(T-t)}{M^{n-l}}
    \left[\sum_{i=1}^{M^{n-l}}
      \left(\funcF_2\big({\bigU}_{l,M}^{(\theta,l,i)}\big)-\1_{\N}(l)\funcF_2\big( {\bigU}_{l-1,M}^{(\theta,-l,i)}\big)\right)
      \left(\uniform_t^{(\theta,l,i)},x+\sppr_{\uniform_t^{(\theta,l,i)}}^{(\theta,l,i)}-\sppr_{t}^{(\theta,l,i)}\right)
    \right].
\end{split}
\end{equation}
\end{setting}

\subsection{An a priori estimate for solutions of partial differential equations (PDEs)}

\begin{lemma}[$q$-th moment of the exact solution]\label{m03d}  Assume \cref{t00} and let 
 $x\in\R^d$, $i\in \{1,2\}$. Then it holds 
that
\begin{equation}
\begin{split}
\sup_{t\in [0,T]}
  \left(\bigg.
    \Exp{\Big.\!\left | \smallU_i(t,x+\fwpr_{t})  \right|^q    }
  \right)^{\!\!\nicefrac{1}{q}}
  \leq
  e^{LT} (T+1)\boundFG
  \left[
  \sup_{t\in [0,T]}
     \left( 
          \Exp{\Big.\!
             \left(1+\normRd{x+\fwpr_t}\Big.\right)^{pq}\bigg.\!
          }\Bigg.\!
    \right)^{ \! \nicefrac{1}{q}} \right]  .
\end{split}\label{m09}
\end{equation}
\end{lemma}
\begin{proof}[Proof of \cref{m03d}]
Throughout this proof 
let $\mu_{t} \colon \mathcal{B}(\R^d) \to [0,1]$, $t \in [0,T]$ be the probability measures which satisfy 
for all $t \in [0,T]$, $B \in \mathcal{B}(\R^d)$ that
\begin{equation}
\label{upper_exact:setting1}
  \mu_t(B) 
=
  \P( x +  \fwpr_t \in B ).
\end{equation}
The integral transformation theorem, \eqref{a01}, and the triangle inequality
show 
for all $t \in [0,T]$  that 
\begin{equation}
\label{upper_exact:eq1}
\begin{split}
  &\left(\bigg. \Exp{\Big. | \smallU_i(t,x +\fwpr_t)|^q } \right)^{ \!\nicefrac{1}{q}}
=
  \left( 
    \int_{\R^d}
      | \smallU_i(t, z)|^q
    \, \mu_t(dz)
  \right)^{ \! \nicefrac{1}{q}} \\
&=
  \left( 
    \int_{\R^d}
      \left| 
        \Exp{
          \funcG_i(z+\fwpr_{T-t}) 
          +
          \int_t^{T}
            (\funcF_i(\smallU_i))(s,z+\fwpr_{s-t})
          \,ds
        }
      \right|^q
    \, \mu_t(dz)
  \right)^{ \! \! \nicefrac{1}{q}} \\
&\leq
  \left( 
    \int_{\R^d}\left|
        \Exp{\big.
            \funcG_i(z+\fwpr_{T-t}) 
        }\right|^q
    \, \mu_t(dz)
  \right)^{ \! \nicefrac{1}{q}} \\
 &\qquad +\int_{t}^T
  \left(
\int_{\R^d}\left| 
       \Exp{  (\funcF_i(\smallU_i))(s,z+\fwpr_{s-t})\big.
                  }      \right|^q\Bigg.\!
    \, \mu_t(dz)\right)^{\!\!\nicefrac{1}{q}} ds
.
\end{split}
\end{equation}
Next,  Jensen's inequality, Fubini's theorem, \eqref{upper_exact:setting1}, the fact that $\fwpr$ has independent and stationary increments, 
and \eqref{m01} demonstrate that 
for all $t \in [0,T]$ it holds that
\begin{equation}
\label{upper_exact:eq3}
\begin{split}
&  \int_{\R^d}\left|
        \Exp{\big.
            \funcG_i(z+\fwpr_{T-t}) 
        }\right|^q\, \mu_t(dz)
\leq    \int_{\R^d}
        \Exp{
          \left| \funcG_i(z+\fwpr_{T}-\fwpr_t)  \right|^q\Big.\!
        }
    \, \mu_t(dz)\\
&
=
        \Exp{
          \left|  \funcG_i\left( x +\fwpr_t+ \fwpr_{T}- \fwpr_t\right)  \right|^q\Big.\!
        }=
        \Exp{
          \left|  \funcG_i\left( x +\fwpr_{T}\right)  \right|^q\Big.\!
        }
\leq \Exp{\boundFG^q\left(1+
          \normRd{  x + \fwpr_{T} }\Big.\!\right)^{pq}\bigg.\!
        }.
\end{split}
\end{equation} 
Furthermore, Jensen's inequality, Fubini's theorem, \eqref{upper_exact:setting1},  the fact that $\fwpr$ has independent and stationary increments, the triangle inequality, \eqref{t01}, and \eqref{m01} demonstrate  
for all $t \in [0,T]$  that
\begin{align} \begin{split}
&\int_{t}^T\left(
\int_{\R^d}\left| 
       \Exp{   \big.\!       
            (\funcF_i(\smallU_i))(s,z+\fwpr_{s-t})           
        }      \right|^q
    \, \mu_t(dz)\right)^{\!\!\nicefrac{1}{q}}ds\\
& \leq 
\int_{t}^T
\left(
\int_{\R^d}
       \Exp{\left| \big.
            (\funcF_i(\smallU_i))(s,z+\fwpr_{s}-\fwpr_{t})
			\right|^q\Big.\!
        }      
    \, \mu_t(dz)\right)^{\!\!\nicefrac{1}{q}} \, ds
\\
&=
\int_{t}^T
\left(\bigg.\!
       \Exp{\left| \big.
            \left(\funcF_i(\smallU_i)\right)(s,x+\fwpr_{t}+\fwpr_{s}-\fwpr_{t})
			\right|^q\Big.\!
        }      \right)^{\!\!\nicefrac{1}{q}}\! ds
\\
&\leq
   \int_{t}^T
     \left( 
          \Exp{\Big.\!
             \left|  (\funcF_i(0))(s, x+\fwpr_{s})  \right|^q
          }\Bigg.\!
    \right)^{ \!\! \nicefrac{1}{q}} \!
  ds
  + 
  \int_{t}^T
     \left( 
          \Exp{\Big.\!
              \left| (\funcF_i(\smallU_i) - \funcF_i(0))(s,x+\fwpr_{s})  \right|^q
          }\Bigg.\!
    \right)^{ \! \nicefrac{1}{q}} 
  \, ds \\
&\leq
  T\sup_{s\in [0,T]}
     \left( 
          \Exp{\Big.\!
             \boundFG^q\left(1+\normRd{x+\fwpr_s}\Big.\right)^{pq}\bigg.\!
          }\Bigg.\!
    \right)^{ \! \nicefrac{1}{q}} 
  + 
  \int_{t}^T
     \left( 
          \Exp{\Big.
              L^q \left| \smallU_i(s,x+\fwpr_{s})  \right|^q
          }\bigg.\!
    \right)^{ \! \nicefrac{1}{q}} 
  \, ds.
\end{split}
\end{align}
Combining this with \eqref{upper_exact:eq1} and \eqref{upper_exact:eq3}
implies that
for all $t \in [0,T]$ it holds that
\begin{equation}
\label{upper_exact:eq8}
\begin{split}
&\left(\bigg. \Exp{\Big.\! \left| \big.\smallU_i(t,x +\fwpr_t)\right|^q } \right)^{\! \!\nicefrac{1}{q}} \\
&\leq
  (T+1)\boundFG\sup_{s\in [0,T]}
     \left( 
          \Exp{\Big.\!
             \left(1+\normRd{x+\fwpr_s}\Big.\right)^{pq}\bigg.\!
          }\Bigg.\!
    \right)^{ \! \nicefrac{1}{q}} 
  + \LipConstF
  \int_{t}^T
     \left( 
          \Exp{\Big.
              \left| \smallU_i(s,x+\fwpr_{s})  \right|^q
          }\bigg.\!
    \right)^{ \! \nicefrac{1}{q}} 
  \, ds.
\end{split}
\end{equation}
Next,
\cite[Corollary 3.11]{HJK+18}
 shows that
\begin{align}
\sup_{s\in [0,T]}\sup_{y\in\R^d} \frac{\left|\smallU_i(s,y)\right|}{\left(1+\normRd{y}\right)^p}\leq 
\sup_{s\in [0,T]}\sup_{y\in\R^d} \frac{\left|\smallU_i(s,y)\right|}{1+\normRd{y}^p}<
\infty.
\end{align}
This, the triangle inequality, 
and the fact that $ \E\!\left[\normRd{\fwpr_T}^{pq}\right]<\infty$ show that 
\begin{align}\label{m03e}\begin{split}
&\int_0^T \left( 
          \Exp{\Big.\!
              \left| \smallU_i(s,x+\fwpr_{s})  \right|^q
          }\bigg.\!
    \right)^{ \!\! \nicefrac{1}{q}} ds\leq \left[\sup_{s\in [0,T]}\sup_{y\in\R^d} \frac{|u(s,y)|}{\left(1+\normRd{y}\right)^p}\right]
\int_{0}^{T}\left(\E\!\left[\left(1+\normRd{x+\fwpr_s}\Big.\!\right)^{pq}\bigg.\!\right]\Bigg.\!\right)^{\!\!\nicefrac{1}{q}}ds\\
&\leq 
\left[\sup_{s\in [0,T]} \sup_{y\in\R^d}\frac{|u(s,y)|}{\left(1+\normRd{y}\right)^p}\right]T
\left(1+\normRd{x}+
\left(\Exp{           \normRd{\fwpr_T}^{pq}\Big.\! }\bigg.\!\right)^{\!\!\frac{1}{pq}}\Bigg.\!\right)^{p}
<\infty.
\end{split}
\end{align}
This,
Gronwall's integral inequality, and \eqref{upper_exact:eq8} establish 
for all $t \in [0, T]$
that
\begin{equation}
\begin{split}
  \left(\bigg.
    \Exp{\Big.\!\left | \smallU_i(t,x+\fwpr_{t})  \right|^q    }
  \right)^{\!\!\nicefrac{1}{q}}\leq
  e^{LT} (T+1)\boundFG\sup_{s\in [0,T]}
     \left( 
          \Exp{\Big.\!
             \left(1+\normRd{x+\fwpr_s}\Big.\right)^{pq}\bigg.\!
          }\Bigg.\!
    \right)^{ \!\! \nicefrac{1}{q}} .
\end{split}
\end{equation}
The proof of \cref{m03d} is thus completed.
\end{proof}

\subsection{A stability result for solutions of PDEs}
\begin{lemma}\label{m04}Assume \cref{t00}.
Then it holds
for all $t\in [0,T]$, $x\in \R^d$  that
\begin{align}\begin{split}
&\E\!\left[\Big.\!\left|\smallU_1(t,x+\fwpr_t)-\smallU_2(t,x+\fwpr_t)\right|\right]\\&\leq\delta 
 \left(e^{LT} (T+1)\right)^{q+1}\left(\boundFG^q+1\right)
         \left(1+\normRd{x}+
\left(\Exp{           \normRd{\fwpr_T}^{pq}\Big.\! }\bigg.\!\right)^{\!\!\frac{1}{pq}}\right)^{pq}.
\end{split}
\end{align}
\end{lemma}
\begin{proof}[Proof of \cref{m04}]
First, \eqref{a01}, the triangle inequality, and the fact that
$\fwpr$ has stationary increments
 show  for all $s\in [0,T]$, $z\in\R^d$  that
\begin{align}\begin{split}
& \left|\smallU_1(s,z)-\smallU_2(s,z)\right|
\\
&=\left|\E\!\left[(\funcG_1-\funcG_2)\left(z+\fwpr_{T-s}\right)+\int_s^T \big(
\funcF_1(\smallU_1)-\funcF_1(\smallU_2)+
\funcF_1(\smallU_2)- \funcF_2(\smallU_2)
\big)\left(t,z+\fwpr_{t-s}\right)\,dt\right]\right|\\
&\leq 
\E\!\left[\Big.\!\left|\big.(\funcG_1-\funcG_2)\left(z+\fwpr_{T}-\fwpr_{s}\right)\right|\right]
+\int_s^T\E\!\left[\left| \big(\funcF_1(\smallU_1)-\funcF_1(\smallU_2)\big)\left(t,z+\fwpr_{t}-\fwpr_{s}\right)\right|\Big.\!\right]\,dt\\
&\qquad\qquad\qquad+\int_s^T\E\!\left[\left| \big(\funcF_1(\smallU_2)- \funcF_2(\smallU_2)\big)\left(t,z+\fwpr_{t}-\fwpr_{s}\right)\right|\bigg.\!\right]\,dt.
\end{split}
\end{align}
This, Fubini's theorem, the fact that $\fwpr$ has independent increments, and the Lipschitz condition in \eqref{t01} ensure that for all $s\in [0,T]$, $x\in \R^d$ it holds that
\begin{align}
\begin{split}
&\E\!\left[\Big.\!\left|\big.\!\left(\smallU_1-\smallU_2\right)(s,x+\fwpr_s)\right|\right]= \E\!\left[\Big.\!\left.\left|\big.\smallU_1(s,z)-\smallU_2(s,z)\right|\right|_{z=x+\fwpr_s}\right]
\\
&\leq \E\!\left[
\E\!\left[\left.\Big.\!\left|\big.(\funcG_1-\funcG_2)\left(z+\fwpr_{T}-\fwpr_{s}\right)\right|\right]\right|_{z=x+\fwpr_s}\right]\\
&\qquad\qquad
+\int_s^T\E\!\left[\E\!\left[\left.\left| \big(\funcF_1(\smallU_1)-\funcF_1(\smallU_2)\big)\left(t,z+\fwpr_{t}-\fwpr_{s}\right)\right|\Big.\!\right]\right|_{z=x+\fwpr_s}\right]\,dt\\&\qquad\qquad+\int_s^T\E\!\left[\E\!\left[\left.\left| \big(\funcF_1(\smallU_2)- \funcF_2(\smallU_2)\big)\left(t,z+\fwpr_{t}-\fwpr_{s}\right)\right|\bigg.\!\right]\right|_{z=x+\fwpr_s}\right]\,dt\\
&= \E\!\left[\Big.\!\left|\big.(\funcG_1-\funcG_2)\left(x+\fwpr_{T}\right)\right|\right]
+\int_s^T\E\!\left[\left| \big(\funcF_1(\smallU_1)-\funcF_1(\smallU_2)\big)\left(t,x+\fwpr_{t}\right)\right|\Big.\!\right]\,dt\\
&\qquad\qquad+\int_s^T\E\!\left[\left| \big(\funcF_1(\smallU_2)- \funcF_2(\smallU_2)\big)\left(t,x+\fwpr_{t}\right)\right|\bigg.\!\right]\,dt\\
&\leq \E\!\left[\Big.\!\left|(\funcG_1-\funcG_2)\left(x+\fwpr_{T}\right)\big.\!\right|\right]
+ \int_s^T\E\!\left[\LipConstF\left| \big(\smallU_1- \smallU_2\big)\left(t,x+\fwpr_{t}\right)\right|\Big.\!\right]\,dt\\
&\qquad\qquad+
T\sup_{t\in [0,T]}\E\!\left[\left| \big(\funcF_1(\smallU_2)- \funcF_2(\smallU_2)\big)\left(t,x+\fwpr_{t}\right)\right|\Big.\!\right].
\end{split}
\end{align}
This, Gronwall's lemma, and \cref{m03d} yield  for all $x\in \R^d$ that
\begin{align}\small\begin{split}
&\sup_{t\in [0,T]}\E\!\left[\Big.\!\left|\big.\!\left(\smallU_1-\smallU_2\right)(t,x+\fwpr_t)\right|\right]\\
&\leq e^{\LipConstF T}(T+1)
\sup_{t\in [0,T]}\max\left\{\E\!\left[\Big.\!\left|\big.(\funcG_1-\funcG_2)\left(x+\fwpr_{T}\right)\right|\right],
\E\!\left[\left| \big(\funcF_1(\smallU_2)- \funcF_2(\smallU_2)\big)\left(t,x+\fwpr_{t}\right)\right|\Big.\!\right]\right\}.
\end{split}\label{m03}
\end{align}
Furthermore, \eqref{m02},  the triangle inequality, and \cref{m03d}
 imply  for all  $x\in \R^d$  that
\begin{align} \begin{split}
&\sup_{t\in [0,T]}\max\left\{\E\!\left[\Big.\!\left|\big.(\funcG_1-\funcG_2)\left(x+\fwpr_{T}\right)\right|\right],
\E\!\left[\left| \big(\funcF_1(\smallU_2)- \funcF_2(\smallU_2)\big)\left(t,x+\fwpr_{t}\right)\right|\Big.\!\right]\right\}
\\
&\leq \delta\sup_{t\in [0,T]}
\E\!\left[\left(\Big.1+\normRd{x+\fwpr_t}\right)^{pq}+\left|u_2(x+\fwpr_t)\right|^q\Big.\!\right]\\
&\leq \delta\sup_{t\in [0,T]}
\E\!\left[\bigg.\!\left(\Big.1+\normRd{x+\fwpr_t}\right)^{pq}\right]
+\delta\sup_{t\in [0,T]}\E\!\left[\big.\!\left|u_2(x+\fwpr_t)\right|^q\Big.\right]
.\\
&\leq \delta\sup_{t\in [0,T]}
\E\!\left[\bigg.\!\left(\Big.1+\normRd{x+\fwpr_t}\right)^{pq}\right]+\delta
 (e^{LT} (T+1)\boundFG)^q\sup_{t\in [0,T]}
          \Exp{\Big.\!
             \left(1+\normRd{x+\fwpr_t}\Big.\!\right)^{pq}\bigg.\!
          }\\
&\leq\delta 
 \left(e^{LT} (T+1)\right)^q (\boundFG^q+1)\sup_{t\in [0,T]}
          \Exp{\Big.\!
             \left(1+\normRd{x+\fwpr_t}\Big.\!\right)^{pq}\bigg.\!
          }.
\end{split}
\end{align}
This, \eqref{m03}, and the triangle inequality  yield that
\begin{align}\begin{split}
&\sup_{t\in [0,T]}\E\!\left[\Big.\!\left|\big.\!\left(\smallU_1-\smallU_2\right)(t,x+\fwpr_t)\right|\right]\\
&\leq\delta 
 \left(e^{LT} (T+1)\right)^{q+1}\left(\boundFG^q+1\right)\sup_{t\in [0,T]}
          \Exp{\Big.\!
             \left(1+\normRd{x+\fwpr_t}\Big.\!\right)^{pq}\bigg.\!
          }
\\
&\leq\delta 
 \left(e^{LT} (T+1)\right)^{q+1}\left(\boundFG^q+1\right)
         \left(1+\normRd{x}+
\left(\Exp{           \normRd{\fwpr_T}^{pq}\Big.\! }\bigg.\!\right)^{\!\!\frac{1}{pq}}\right)^{pq}.
\end{split}\label{m10}
\end{align}
This
completes the proof of \cref{m04}.
\end{proof}
%%%

%%%

\subsection{A stability result for MLP approximations}
\begin{corollary}\label{m06}Assume \cref{t00}, let $x\in\R^d$,  $N,M\in \N$, and assume that $q\geq 2$.
Then it holds that
\begin{align}\begin{split}
&\left(\bigg.\!\E\!\left[\left|U^0_{N,M}(0,x)-\smallU_1(0,x)\right|^2\right]\right)^{\!\!\nicefrac{1}{2}}\\
&\leq 
  \left(e^{LT} (T+1)\right)^{q+1}\left(\boundFG^q+1\right)\left(\delta+\frac{e^{M/2}(1+2LT)^{N}}{M^{N/2}}\right)
  \left(1+\normRd{x}+
\left(\Exp{           \normRd{\fwpr_T}^{pq}\Big.\! }\bigg.\!\right)^{\!\!\frac{1}{pq}}\right)^{pq}.
\end{split}
\end{align}
\end{corollary}
\begin{proof}[Proof of \cref{m06}]
\begin{comment}

First, observe that the fact that $\forall \, x\in [0,\infty)\colon 1+x^{p}\leq (1+x)^{p}$ and the triangle inequality show for all $t\in[0,T]$ that
\begin{align}\begin{split}
\!\!\left(\Exp{\Big.\!            \left(1+\normRd{x+\fwpr_t}^{p}\Big.\right)^2 }\Bigg.\!\right)^{\!\!\frac{1}{2p}}
\leq &
\left(\Exp{\Big.\!            \left(1+\normRd{x+\fwpr_t}\Big.\right)^{2p} }\Bigg.\!\right)^{\!\!\frac{1}{2p}}\leq 1+
\left(\Exp{           \normRd{x+\fwpr_t}^{pq}\Big.\! }\bigg.\!\right)^{\!\!\frac{1}{2p}}\\
\leq& 1+\normRd{x}+
\left(\Exp{           \normRd{\fwpr_t}^{2p}\Big.\! }\bigg.\!\right)^{\!\!\frac{1}{2p}}
\leq 1+\normRd{x}+
\left(\Exp{           \normRd{\fwpr_T}^{2p}\Big.\! }\bigg.\!\right)^{\!\!\frac{1}{2p}}
.
\end{split}
\end{align}
Hence, it holds for all $t\in [0,T]$ that
\begin{align}\begin{split}
\left(\Bigg.\!\Exp{\Big.\!            \left(1+\normRd{x+\fwpr_t}^{p}\Big.\right)^2\bigg.\! }\right)^{\!\!\nicefrac{1}{2}}
\leq 
\left(1+\normRd{x}+
\left(\Exp{           \normRd{\fwpr_T}^{2p}\Big.\! }\right)^{\!\!\frac{1}{2p}}\right)^{p}.
\end{split}
\end{align}
\end{comment}
First, \cref{m03d} implies that
$\int_{0}^{T}
\left(
    \Exp{\big.\!\left | \smallU_i(t,x+\fwpr_{t})  \right|^2    }
  \right)^{\!\!\nicefrac{1}{2}}dt<\infty$. 
This, \cite[Theorem~3.5]{HJK+18} 
(with $\xi=x$,  $F=\funcF_2$, $g=\funcG_2$, and $u=\smallU_2$ in the notation of \cite[Theorem~3.5]{HJK+18}),  \eqref{m01}, and the triangle inequality
 ensure  that
\begin{align}\begin{split}
&\left(\bigg.\!\E\!\left[\left|U^0_{N,M}(0,x)-\smallU_2(0,x)\right|^2\right]\right)^{\!\!\nicefrac{1}{2}}\\
&\leq
  e^{\LipConstF T}
  \left[\left(\E\!\left[\Big.\!\left|\big.\funcG_2(x+\fwpr_T)\right|^2\right]\bigg.\!\right)^{\!\!\nicefrac{1}{2}}
+T\left(\frac{1}{T}\int_0^T\E\!\left[\Big.\!\left|(\funcF_2(0))(t,x+\fwpr_t)\right|^2\right]dt\Bigg.\! \right)^{\!\!\nicefrac{1}{2}} \right]
  \frac{e^{M/2}(1+2LT)^{N}}{M^{N/2}}\\
&\leq e^{\LipConstF T} (T+1) \sup_{t\in [0,T]}
\left(
\E\!\left[\boundFG^2 \left(1+\normRd{x+\fwpr_t}\Big.\!\right)^{2p}\right]\Bigg.\!\right)^{\!\!\nicefrac{1}{2}}  \frac{e^{M/2}(1+2LT)^{N}}{M^{N/2}}
\\
&\leq e^{\LipConstF T} (T+1)\boundFG 
\left(1+\normRd{x}+
\left(\Exp{           \normRd{\fwpr_T}^{2p}\Big.\! }\right)^{\!\!\frac{1}{2p}}\Bigg.\!\right)^{p} \frac{e^{M/2}(1+2LT)^{N}}{M^{N/2}}.
\end{split}\label{m11}
\end{align}
Furthermore, \cref{m04} shows that
\begin{align}\begin{split}
\left|u_2(0,x)-u_1(0,x)\right|
&\leq
\delta 
 \left(e^{LT} (T+1)\right)^{q+1}\left(\boundFG^q+1\right)
         \left(1+\normRd{x}+
\left(\Exp{           \normRd{\fwpr_T}^{pq}\Big.\! }\bigg.\!\right)^{\!\!\frac{1}{pq}}\right)^{pq}.
\end{split}
\end{align}
This, the triangle inequality, \eqref{m11}, the fact that
$B\leq B^q+1$,
the assumption that $q\geq 2$, and
Jensen's inequality show that
\begin{align}\begin{split}
&\left(\bigg.\!\E\!\left[\left|U^0_{N,M}(0,x)-\smallU_1(0,x)\right|^2\right]\right)^{\!\!\nicefrac{1}{2}}\\
&\leq 
\left(\bigg.\!\E\!\left[\left|U^0_{N,M}(0,x)-\smallU_2(0,x)\right|^2\right]\right)^{\!\!\nicefrac{1}{2}}+\left|\big.u_2(0,x)-u_1(0,x)\right|\\
&\leq 
  \left(e^{LT} (T+1)\right)^{q+1}\left(\boundFG^q+1\right)\left(\delta+\frac{e^{M/2}(1+2LT)^{N}}{M^{N/2}}\right)\!\!
  \left(1+\normRd{x}+
\left(\Exp{           \normRd{\fwpr_T}^{pq}\Big.\! }\bigg.\!\right)^{\!\!\frac{1}{pq}}\right)^{pq}.
\end{split}
\end{align}
The proof of \cref{m06} is thus completed.
\end{proof}

\section{Deep neural network representations for MLP approximations}\label{sec:Picard}
The main result of this section, \cref{b26} below, shows that multilevel Picard aproximations can be well represented by DNNs. The central tools for the proof of \cref{b26} are \cref{m11b,b01} which show that DNNs are stable under compositions and summations. We formulate \cref{m11b,b01} in terms of the operators defined in \eqref{b26c} below, whose properties are studied in Lemmas~\ref{d01},~\ref{d02},~and \ref{b15}.
\subsection{A mathematical framework for deep neural networks}
\begin{setting}[Artificial neural networks]\label{m07}
Let $\left \|\cdot \right\|, \supnorm{\cdot} \colon (\cup_{d\in \N} \R^d) \to [0,\infty)$ and 
$\cardna \colon (\cup_{d\in \N}\R^d) \to \N$
satisfy for all $d\in \N$, $x=(x_1,\ldots,x_d)\in \R^d$ that $\|x\|=\sqrt{\sum_{i=1}^d(x_i)^2}$, $\supnorm{x}=\max_{i\in [1,d]\cap \N}|x_i|$, and
$\card{x}=d$,
let
 $\funcBoldA{d}\colon \R^d\to\R^d $, $d\in \N$, satisfy for all $d\in\N$, $x=(x_1,\ldots,x_d)\in \R^d$ that 
\begin{align}
\funcBoldA{d}(x)= \left(\max\{x_1,0\},\ldots,\max\{x_d,0\}\right),
\end{align}
let $\setCalD=\cup_{H\in \N} \N^{H+2}$,
let
\begin{align}\begin{split}
\setCalN= \bigcup_{H\in  \N}\bigcup_{(k_0,k_1,\ldots,k_{H+1})\in \N^{H+2}}
\left[ \prod_{n=1}^{H+1} \left(\R^{k_{n}\times k_{n-1}} \times\R^{k_{n}}\right)\right],
%\quad\text{and}\quad
%\setCalD=\bigcup_{H\in \N} \N^{H+2},
\end{split}
\end{align}
let $\funcCalL\colon \setCalN\to\setCalD$ and
$\funcCalR\colon \setCalN\to (\cup_{k,l\in \N} C(\R^k,\R^l))$
%\begin{align}
%\funcCalL\colon \setCalN\to\setCalD\quad\text{and}\quad
%\funcCalR\colon \setCalN\to \bigcup_{k,l\in \N} C(\R^k,\R^l)\label{m07b}
%\end{align}
satisfy
for all $H\in \N$, $k_0,k_1,\ldots,k_H,k_{H+1}\in \N$,
$
\Phi = ((W_1,B_1),\ldots,(W_{H+1},B_{H+1}))\in \prod_{n=1}^{H+1} \left(\R^{k_n\times k_{n-1}} \times\R^{k_n}\right), 
$
$x_0 \in \R^{k_0},\ldots,x_{H}\in \R^{k_{H}}$ with 
$\forall\, n\in \N\cap [1,H]\colon x_n = \funcBoldA{k_n}(W_n x_{n-1}+B_n )$ 
that
\begin{equation}\label{m07c}
\funcCalL(\Phi)= (k_0,k_1,\ldots,k_{H}, k_{H+1}),\qquad
\funcCalR(\Phi )\in C(\R^{k_0},\R ^ {k_{H+1}}),
\end{equation}
\begin{equation}\label{m07d}
\text{and}\qquad (\funcCalR(\Phi)) (x_0) = W_{H+1}x_{H}+B_{H+1},
\end{equation} 
let $\concat \colon \setCalD\times \setCalD \to\setCalD $ satisfy 
for all $H_1,H_2\in \N$, $ \alpha=(\alpha_0,\alpha_1,\ldots,\alpha_{H_1},\alpha_{H_1+1})\in\N^{H_1+2}$, $\beta=(\beta_0,\beta_1,\ldots,\beta_{H_2},\beta_{H_2+1})\in\N^{H_2+2}$
%\begin{align}
%H_1,H_2\in \N, \quad \alpha=(\alpha_0,\alpha_1,\ldots,\alpha_{H_1},\alpha_{H_1+1})\in\N^{H_1+2},\quad 
%\beta=(\beta_0,\beta_1,\ldots,\beta_{H_2},\beta_{H_2+1})\in\N^{H_2+2}
%\end{align}
that
\begin{align}
\alpha\concat\beta= (\beta_{0},\beta_{1},\ldots,\beta_{H_2},\beta_{H_2+1}+\alpha_{0},\alpha_{1},\alpha_{2},\ldots,\alpha_{H_1+1})\in \N^{H_1+H_2+3},\label{b25}
\end{align}
let $\dimsum \colon \setCalD\times \setCalD \to\setCalD  $ satisfy
for all $H\in \N$, 
$\alpha= (\alpha_0,\alpha_1,\ldots,\alpha_{H},\alpha_{H+1})\in \N^{H+2}$,
$\beta= (\beta_0,\beta_1,\beta_2,\ldots,\beta_{H},\beta_{H+1})\in \N^{H+2}$
%\begin{align}
%\dimsum\colon \bigcup_{H, k,l \in \N}\left(\{k\}\times \N^{H}\times \{l\}\right)^2 \to 
%\bigcup_{H, k,l \in \N}\left(
%\{k\}\times \N^{H} \times \{l\}\right)\label{b26a}
%\end{align} be 
%the function  which satisfies that for all
%\begin{align}\begin{split}
%H\in \N,\quad 
%\alpha= (\alpha_0,\alpha_1,\ldots,\alpha_{H},\alpha_{H+1}),\,
%\beta= (\beta_0,\beta_1,\beta_2,\ldots,\beta_{H},\beta_{H+1})\in \N^{H+2}\label{b26b}
%\end{split}
%\end{align}
that
\begin{align}
\alpha \dimsum \beta =(\alpha_0,\alpha_1+\beta_1,\ldots,\alpha_{H}+\beta_{H},\beta_{H+1})\in \N^{H+2},\label{b26c}
\end{align} 
and let
$\neutralDim{n}\in \setCalD$, $n\in [3,\infty)\cap\N$, satisfy for all $n\in [3,\infty)\cap\N$ that
\begin{align}
 \neutralDim{n}= (1,\underbrace{2,\ldots,2}_{(n-2)\text{-times}},1)\in \N^{n}.\label{b26d}
\end{align}
\end{setting}
\begin{remark}
The set $\setCalN$ can be viewed as the set of all artificial neural networks. 
%The set $\setCalD$ is called the set of dimension structures. 
For each network $\Phi\in \setCalN$ the function 
$\funcCalR(\Phi)$ is the function represented by $\Phi$ and
the vector $\funcCalL(\Phi)$ describes the layer dimensions of $\Phi$. 
\end{remark}

\subsection{Properties of operations associated to deep neural networks}

\begin{lemma}[$\concat$ is associative]\label{d01}Assume \cref{m07} and let $\alpha,\beta,\gamma\in \setCalD$. Then it holds that
$(\alpha\concat\beta)\concat \gamma
= \alpha\concat(\beta\concat \gamma)$.
\end{lemma}
\begin{proof}[Proof of \cref{d01}]Throughout this proof let 
$H_1,H_2,H_3\in \N$,
let $(\alpha_i)_{i\in [0,H_1+1]\cap\N_0}\in \N^{H_1+2}$,
$(\beta_i)_{i\in [0,H_2+1]\cap\N_0}\in \N^{H_2+2}$, 
$(\gamma_i)_{i\in [0,H_3+1]\cap\N_0}\in \N^{H_3+2}$ satisfy that
\begin{align}\begin{split}
\alpha&=(\alpha_0,\alpha_1,\ldots,\alpha_{H_1+1}),\quad 
\beta=(\beta_0,\beta_1,\ldots,\beta_{H_2+1}),\quad\text{and}\\
\gamma&=(\gamma_0,\gamma_1,\ldots,\gamma_{H_3+1}).
\end{split}
\end{align}
The definition of $\concat$ in \eqref{b25} then shows that
\begin{align}\begin{split}
(\alpha\concat\beta)\concat \gamma&=
(\beta_{0},\beta_{1},\beta_2\ldots,\beta_{H_2},\beta_{H_2+1}+\alpha_{0},\alpha_{1},\alpha_{2},\ldots,\alpha_{H_1+1})\concat (\gamma_0,\gamma_1,\ldots,\gamma_{H_3+1})\\
&=
(\gamma_0,\ldots,\gamma_{H_3},\gamma_{H_3+1}+\beta_{0},\beta_{1},\ldots,\beta_{H_2},\beta_{H_2+1}+\alpha_{0},\alpha_{1},\alpha_{2},\ldots,\alpha_{H_1+1})\\
&=
(\alpha_0,\alpha_1,\ldots,\alpha_{H_1+1})\concat 
(\gamma_{0},\gamma_{1},\ldots,\gamma_{H_3},\gamma_{H_3+1}+\beta_{0},\beta_{1},\beta_{2},\ldots,\beta_{H_2+1})
\\
&=\alpha\concat (\beta\concat\gamma).
\end{split}
\end{align}
The proof of \cref{d01} is thus completed.
\end{proof}
\begin{lemma}[$\dimsum$ and associativity]\label{d02}Assume \cref{m07},
let $H,k,l \in \N$, and let $\alpha,\beta,\gamma\in \left( \{k\}\times \N^{H} \times \{l\}\right)$.
Then
\begin{enumerate}[(i)]
\item it holds that $\alpha\dimsum\beta\in \left(\{k\}\times \N^{H} \times \{l\}\right)$,
\item it holds that $\beta\dimsum \gamma\in \left(\{k\}\times \N^{H} \times \{l\}\right)$, and 
\item it holds that $(\alpha\dimsum\beta)\dimsum \gamma
= \alpha\dimsum(\beta\dimsum \gamma)$.
\end{enumerate}
\end{lemma}
\begin{proof}[Proof of \cref{d02}]
Throughout this proof let $\alpha_i,\beta_i,\gamma_i\in \N$, $i\in [1,H]\cap\N$, satisfy that
$
\alpha= (k,\alpha_1,\alpha_2,\ldots,\alpha_{H},l)$,
$\beta= (k,\beta_1,\beta_2,\ldots,\beta_{H},l)$, and
$\gamma= (k,\gamma_1,\gamma_2,\ldots,\gamma_{H},l).$
The definition of $\dimsum$ (see \eqref{b26c}) then shows that
\begin{align}\begin{split}
\alpha \dimsum \beta 
&=(k,\alpha_1+\beta_1,
\alpha_2+\beta_2,
\ldots,\alpha_{H}+\beta_{H},l)\in \{k\}\times\N^{H}\times\{l\},
\\
\beta\dimsum\gamma 
&=(k,\beta_1+\gamma_1,
\beta_2+\gamma_2,
\ldots,\beta_{H}+\gamma_{H},l)\in \{k\}\times\N^{H}\times\{l\},
\end{split}
\end{align} and
\begin{align}\begin{split}
(\alpha \dimsum \beta)\dimsum\gamma
&=(k,(\alpha_1+\beta_1)+\gamma_1,
(\alpha_2+\beta_2)+\gamma_2,
\ldots,(\alpha_{H}+\beta_{H})+\gamma_{H},l)
\\&=(k,\alpha_1+(\beta_1+\gamma_1),
\alpha_2+(\beta_2+\gamma_2),
\ldots,\alpha_{H}+(\beta_{H}+\gamma_{H}),l)
=
\alpha \dimsum (\beta\dimsum\gamma).\end{split}
\end{align} 
%This implies that
%$(\alpha\dimsum\beta)\dimsum \gamma
%= \alpha\dimsum(\beta\dimsum \gamma)$.
The proof of \cref{d02} is thus completed.
\end{proof}

\begin{lemma}[Triangle inequality]\label{b15}
Assume \cref{m07},
let $H,k,l \in \N$, and let $\alpha,\beta\in \{k\}\times \N^{H} \times \{l\}$.
Then it holds that
$\supnorm{\alpha\dimsum\beta}\leq\supnorm{\alpha}+
\supnorm{\beta} $.
\end{lemma}
\begin{proof}[Proof of \cref{b15}]
Throughout this proof let $\alpha_i,\beta_i\in \N$, $i\in [1,H]\cap\N$ satisfy that
$
\alpha= (k,\alpha_1,\alpha_2,\ldots,\alpha_{H},l)
$
and
$
\beta= (k,\beta_1,\beta_2,\ldots,\beta_{H},l).
$
The definition of $\dimsum$ (see \eqref{b26c}) then shows that
$
\alpha \dimsum \beta =(k,\alpha_1+\beta_1,
\alpha_2+\beta_2,
\ldots,\alpha_{H}+\beta_{H},l).
$
This together with the triangle inequality implies that
\begin{align}\begin{split}
\supnorm{\alpha \dimsum \beta }&=\sup\left\{|k|,\left|\alpha_1+\beta_1\right|,
\left|\alpha_2+\beta_2\right|,
\ldots,\left|\alpha_{H}+\beta_{H}\right|,\left|l\right|\right\}\\
&\leq 
\sup\left\{|k|,\left|\alpha_1\right|,
\left|\alpha_2\right|,
\ldots,\left|\alpha_{H}\right|,\left|l\right|\right\}
+\sup\left\{|k|,\left|\beta_1\right|,
\left|\beta_2\right|,
\ldots,\left|\beta_{H}\right|,\left|l\right|\right\}\\
&= \supnorm{\alpha}+\supnorm{\beta}.\end{split}
\end{align} 
This completes the proof of \cref{b15}.
\end{proof}
The following result, \cref{b03}, is a variant of \cite[Lemma 5.4]{JSW18}.
\begin{lemma}[Existence of DNNs with $H\in\N$ hidden layers for the identity in $\R$]\label{b03}
Assume \cref{m07} and let $H\in \N$.
Then it holds that
$\identity{\R}\in \funcCalR(\{\Phi\in\setCalN\colon\funcCalL(\Phi)=\neutralDim{H+2} \}) $.
\end{lemma}
\begin{proof}[Proof of \cref{b03}]
Throughout this proof 
let $W_1\in \R^{2\times1}$, $W_i\in \R^{2\times2}$, $\,i\in [2,H]\cap\N $,
$W_{H+1}\in\R^{1\times2}$, 
$B_i\in\R^2$, $i\in [1,H]\cap\N$,
$B_{H+1}\in\R^1$ satisfy that
\begin{align}\begin{split}
&W_1= \begin{pmatrix}
1\\
-1
\end{pmatrix},\quad
\forall i\in [2,H]\cap\N\colon
W_i=\begin{pmatrix}
1& 0\\
0& 1
\end{pmatrix}
,
\quad 
W_{H+1}= \begin{pmatrix}
1&-1
\end{pmatrix},\\
&
\forall i\in [1,H]\cap\N\colon 
B_i= \begin{pmatrix}
0\\0
\end{pmatrix},\quad B_{H+1}=0,\end{split}
\label{b03b}
\end{align}
let
$\phi\in \setCalN$ satisfy that
 $\phi=((W_1,B_1),(W_2,B_2),\ldots,(W_H,B_H),(W_{H+1},B_{H+1}))$, 
for every $a\in\R$ let $a^+\in [0,\infty)$ be the non-negative part of $a$, i.e., $a^+=\max\{a,0\}$, and 
let
$x_0\in \R$, $x_1,x_2,\ldots,x_{H}\in\R^2$ satisfy for all
$ n\in \N\cap [1,H]$ that 
\begin{align} 
x_n = \funcBoldA{2}(W_n x_{n-1}+B_n ).\label{b03c}
\end{align}
Note that \eqref{b03b} and the definition of  $\funcCalL$ (see \eqref{m07c}) imply that
$\funcCalL(\phi)=\neutralDim{H+2}$.
Furthermore,  \eqref{b03b}, \eqref{b03c}, and an induction argument show  that
\begin{align}
\begin{split}
x_1&= \funcBoldA{2}(W_1x_0+B_1)= 
\funcBoldA{2}\left(\begin{pmatrix}
x_0\\
-x_0
\end{pmatrix}\right)=
\begin{pmatrix}
x_0^+\\
(-x_0)^{+}
\end{pmatrix},\\
x_2&= \funcBoldA{2}(W_2x_1+B_2)= 
\funcBoldA{2}(x_1)=\funcBoldA{2}\left(\begin{pmatrix}
x_0^+\\
(-x_0)^{+}
\end{pmatrix}\right)=
\begin{pmatrix}
x_0^+\\
(-x_0)^{+}
\end{pmatrix}
,\\
&\quad \vdots\\
x_{H}&= \funcBoldA{2}(W_{H}x_{H-1}+B_{H})= 
\funcBoldA{2}(x_{H-1})=\funcBoldA{2}\left(\begin{pmatrix}
x_0^+\\
(-x_0)^{+}
\end{pmatrix}\right)=
\begin{pmatrix}
x_0^+\\
(-x_0)^{+}
\end{pmatrix}.
\end{split}
\end{align}
The definition of $\funcCalR$ (see \eqref{m07d}) hence ensures that
\begin{align}
(\funcCalR(\phi))(x_0)&=W_{H+1}x_{H}+B_{H+1}= 
\begin{pmatrix}
1&-1
\end{pmatrix}
\begin{pmatrix}
x_0^+\\
(-x_0)^{+}
\end{pmatrix}=x_0^{+}-(-x_0)^{+}=x_0.
\end{align}
The fact that $x_0$ was arbitrary therefore
proves
 that $ \funcCalR(\phi) =\identity{\R}$. This and the fact that $\funcCalL(\phi)=\neutralDim{H+2}$ demonstrate that
$\identity{\R}\in \funcCalR(\{\Phi\in\setCalN\colon\funcCalL(\Phi)=\neutralDim{H+2} \}) $.
The proof of \cref{b03} is thus completed.
\end{proof}
\begin{lemma}[DNNs for affine transformations]\label{p01}
Assume \cref{m07} and let $d,m\in \N$, 
$\lambda\in \R$,
$b\in\R^d$, $a\in\R^m$, $\Psi\in\setCalN$ satisfy that $\funcCalR(\Psi)\in C(\R^d,\R^m)$. Then it holds that
\begin{align}
\lambda\left(\Big.\!\big(\funcCalR(\Psi)\big)(\cdot +b)+a\right)\in \funcCalR\Big(\{\Phi\in\setCalN\colon \funcCalL(\Phi)=\funcCalL(\Psi)\}\Big).
\end{align}
\end{lemma}
\begin{proof}[Proof of \cref{p01}]
Throughout this proof let $H,k_0,k_1,\ldots,k_{H+1}\in\N$ satisfy that 
\begin{align}
H+2=\card{\funcCalL(\Psi)}  \quad\text{and}\quad (k_0,k_1,\ldots,k_{H},k_{H+1}) = \funcCalL(\Psi),
\end{align}
let $((W_1,B_1),(W_2,B_2),\ldots,(W_H,B_H),(W_{H+1},B_{H+1})) \in \prod_{n=1}^{H+1}\left(\R^{k_n\times k_{n-1}}\times \R^{k_n}\right)$ satisfy that
\begin{align}
\Big((W_1,B_1),(W_2,B_2),\ldots,(W_H,B_H),(W_{H+1},B_{H+1})\Big)=\Psi, 
\end{align}
let $\phi\in \setCalN$ satisfy that 
\begin{align}
\phi=\Big((W_1,B_1+W_1b),(W_2,B_2),\ldots,(W_H,B_H),(\lambda W_{H+1},\lambda B_{H+1}+\lambda a)\Big),
\end{align}
and let
$x_0,y_0 \in \R^{k_0},x_1,y_1 \in \R^{k_1},\ldots,x_{H},y_H\in \R^{k_{H}}$ satisfy 
for all $n\in \N\cap [1,H]$
that
\begin{align} 
x_n = \funcBoldA{k_n}(W_n x_{n-1}+B_n ),\, 
y_n = \funcBoldA{k_n}(W_n y_{n-1}+B_n+\1_{\{1\}}(n)W_1b )
\quad\text{and} \quad x_0=y_0+b.
\end{align}
Then
it holds that
\begin{align}
y_1= \funcBoldA{k_1}(W_1 y_{0}+B_1+W_1b )= \funcBoldA{k_1}(W_1( y_{0}+b)+B_1 )
=\funcBoldA{k_1}(W_1x_0+B_1 )=x_1.
\end{align}
This and an induction argument prove for all $ i\in [2,H]\cap\N$ that
\begin{align}\begin{split}
y_i=\funcBoldA{k_i}(W_i y_{i-1}+B_i )= \funcBoldA{k_i}(W_i x_{i-1}+B_i )=x_i.
\end{split}
\end{align}
The definition of $\funcCalR$ (see \eqref{m07d}) hence shows that
\begin{align}\begin{split}
(\funcCalR(\phi))(y_0)&= \lambda W_{H+1}y_H+\lambda B_{H+1}+\lambda a=\lambda W_{H+1}x_H+\lambda B_{H+1}+\lambda a\\&
=\lambda (W_{H+1}x_H+ B_{H+1}+ a)
=\lambda(
(\funcCalR(\Psi))(x_0)+a)= \lambda (\funcCalR(\Psi))(y_0+b)+a.
\end{split}
\end{align}
This and the fact that
$y_0$ was arbitrary
prove that $\funcCalR(\phi)=\lambda ((\funcCalR(\Psi))(\cdot+b)+a)
$. This and the fact that
$\funcCalL(\phi)=\funcCalL(\Psi)$ imply that
$\lambda\left((\funcCalR(\Psi))(\cdot +b)+a\right)\in \funcCalR(\{\Phi\in\setCalN\colon \funcCalL(\Phi)=\funcCalL(\Psi)\})$. The proof of \cref{p01} is thus completed.
\end{proof}
\begin{lemma}[Composition]\label{m11b}
Assume \cref{m07} and let $d_1,d_2,d_3\in\N$, $f\in C(\R^{d_2},\R^{d_3})$, $g\in C(  \R^{d_1}, \R^{d_2}) $, 
$\alpha,\beta\in \setCalD$ satisfy that
$f\in \funcCalR(\{\Phi\in \setCalN\colon \funcCalL(\Phi)=\alpha\})$
and
$g\in \funcCalR(\{\Phi\in \setCalN\colon \funcCalL(\Phi)=\beta\})$.
Then it holds
that $(f\circ g)\in \funcCalR(\{\Phi\in \setCalN\colon \funcCalL(\Phi)=\alpha\concat\beta\})$.
\end{lemma}
\begin{proof}[Proof of \cref{m11b}]
Throughout this proof let $H_1,H_2,\alpha_0,\ldots, \alpha_{H_1+1},\beta_0,\ldots, \beta_{H_2+1}\in \N$, $\funcPhi{f}, \funcPhi{g}\in \setCalN$ satisfy that
\begin{equation}
\begin{split}
&(\alpha_0,\alpha_1,\ldots,\alpha_{H_1+1})=\alpha, \quad
(\beta_0,\beta_1,\ldots,\beta_{H_2+1})=\beta, \quad
\funcCalR(\funcPhi{f})=f , \\
&\funcCalL(\funcPhi{f})=\alpha , \quad
 \funcCalR(\funcPhi{g})=g,\quad \text{and}\quad
\funcCalL(\funcPhi{g})=\beta. 
\end{split}
\end{equation}
Lemma 5.4 in \cite{JSW18} shows that there exists $\mathbb{I}\in\setCalN$ such that $\funcCalL(\mathbb{I})=d_2\neutralDim{3}= (d_2,2d_2,d_2)$
and $\funcCalR(\mathbb{I}) =\identity{\R^{d_2}}$. 
Note that $2d_2=\beta_{H_2+1}+\alpha_0$.
This and
\cite[Proposition 5.2]{JSW18}
(with $\phi_1= \funcPhi{f}$, $\phi_2= \funcPhi{g}$, and $\mathbb{I}=\mathbb{I}$ in the notation of \cite[Proposition 5.2]{JSW18})
 show that there exists 
$\funcPhi{f\circ g}\in\setCalN$ such that
$  \funcCalR ( \funcPhi{f\circ g})=f\circ g$ and $\funcCalL(\funcPhi{f\circ g})= \funcCalL(\funcPhi{f})\concat\funcCalL(\funcPhi{g})=\alpha\concat\beta$. Hence,
it holds
that $f\circ g\in \funcCalR(\{\Phi\in \setCalN\colon \funcCalL(\Phi)=\alpha\concat\beta\})$.
The proof of \cref{m11b} is thus completed.
\end{proof}
The following result, \cref{b01}, essentially generalizes \cite[Lemma 5.1]{JSW18} to the case where the DNNs have different hidden layer dimensions.
% structures in the hidden layers.
\begin{lemma}[Sum of DNNs of the same length]
\label{b01}
Assume \cref{m07} and let $M,H,p,q\in \N$,  $h_1,h_2,\ldots,h_M\in\R$,
 $k_i\in \setCalD $,
$f_i\in C(\R^{p},\R^{q})$,
$i\in [1,M]\cap\N$, satisfy 
for all $i\in [1,M]\cap\N$
that
\begin{align}
 \!\card{k_i}=H+2\quad\text{and}\quad f_i\in 
\funcCalR\left(\Big.\!\left\{\big.\Phi\in\setCalN\colon \funcCalL(\Phi)=k_i\right\}\right).
\end{align}
Then
it holds that 
\begin{align}
\sum_{i=1}^{M}h_if_i
\in\funcCalR\left(\left\{ \Phi\in\setCalN\colon
\funcCalL(\Phi)=\dimsum_{i=1}^Mk_i\right\}\right).
\label{b07}
\end{align}
\end{lemma}

\begin{proof}[Proof of \cref{b01}]
Throughout this proof 
let 
$\phi_i\in \setCalN $, $i\in [1,M]\cap\N$, and
$k_{i,n}\in\N$, $i\in [1,M]\cap\N$, $n\in [0,H+1]\cap\N_0$, satisfy for all
$ i \in [1,M]\cap\N$ that 
\begin{align}
 \funcCalL(\phi_i)=k_i=
(k_{i,0},k_{i,1},k_{i,2},\ldots,k_{i,H},k_{i,H+1})
 \quad\text{and}\quad\funcCalR(\phi_i)=f_i,\label{b08a}
\end{align}
for every $i\in[1,M]\cap\N$
let $((W_{i,1}, B_{i,1}),\ldots,  (W_{i,H+1}, B_{i,H+1}))\in \prod_{n=1}^{H+1}\left(\R^{k_{i,n}\times k_{i,n-1}} \times \R^{k_{i,n}}\right)$
satisfy that  
\begin{align}
\phi_i=
\left((W_{i,1}, B_{i,1}),\ldots,  (W_{i,H+1},B_{i,H+1})\right) ,
\end{align}
let 
$k_n^{\dimsum}\in\N$, $n\in [1,H]\cap\N$,
$k^{\dimsum}\in \N^ {H+2}$
satisfy for all $n\in [1,H]\cap\N$ that
\begin{align}\label{b01b}\begin{split}
k_n^{\dimsum}=\sum_{i=1}^{M}k_{i,n}\quad\text{and}\quad 
k^{\dimsum}=(p,k^{\dimsum}_1,k^{\dimsum}_2,\ldots, k^{\dimsum}_{H},q),
\end{split}
\end{align}
let $W_1\in \R^{k_1^{\dimsum}\times p}$, $B_1\in \R^{k_1^{\dimsum}}$ satisfy that
\begin{align}
W_1=
\begin{pmatrix}
W_{1,1}\\
W_{2,1}\\
\vdots\\
W_{M,1}
\end{pmatrix}
\quad\text{and}\quad
B_1=
\begin{pmatrix}
B_{1,1}\\
B_{2,1}\\
\vdots\\
B_{M,1}
\end{pmatrix},\label{b01c}
\end{align}
let $W_n\in\R^{k_n^{\dimsum}\times k_{n-1}^{\dimsum}}$, $B_n\in\R^{k^{\dimsum}_{n}}$, $n\in [2,H]\cap\N$, satisfy for all $n\in [2,H]\cap\N$ that
\begin{align}\begin{split}
W_n= \begin{pmatrix}
W_{1,n} &	0		&	0		&	0	\\
0 		&	W_{2,n}	&	0		&	0	\\
0		&	0		&	\ddots	&	0	\\
0		&	0		&	0		&W_{M,n}
\end{pmatrix}
\quad\text{and}\quad
B_n=
\begin{pmatrix}
B_{1,n}\\
B_{2,n}\\
\vdots\\
B_{M,n}
\end{pmatrix},\end{split}\label{b01d}
\end{align}
let $W_{H+1}\in  \R^{q\times k_{H}^{\dimsum}}$, $B_{H+1}\in \R^{q}$ satisfy that
\begin{align}\begin{split}
W_{H+1}= \begin{pmatrix}
h_1W_{1,H+1}&\ldots&h_MW_{M,H+1}
\end{pmatrix}\quad\text{and}\quad B_{H+1} = \sum_{i=1}^{M}h_iB_{i,H+1},
\end{split}\label{b01e}
\end{align}
let $x_0\in\R^p,\, x_1\in \R^{k_1^{\dimsum}},
x_2\in \R^{k_2^{\dimsum}}
\ldots,x_H\in \R^{k_H^{\dimsum}}$, 
let 
$x_{1,0},x_{2,0},\ldots,x_{M,0}\in \R^{p}$,
 $x_{i,n}\in \R^{k_{i,n}}$, $i\in [1,M]\cap\N$, $n\in [1,H]\cap\N$, satisfy 
for all $i\in [1,M]\cap\N$, $n\in [1,H]\cap\N$
that 
\begin{align}\begin{split}
&x_0=x_{1,0}=x_{2,0}=\ldots=x_{M,0},\\
&x_{i,n}=\funcBoldA{k_{i,n}}(W_{i,n}x_{i,n-1}+B_{i,n}),\\ 
&x_n= \funcBoldA{k^{\dimsum}_{n}}(W_{n}x_{n-1}+B_{n}),
\end{split}
\end{align}
and let $\psi\in \setCalN$ satisfy that
\begin{align}
\psi= \left((W_1,B_1),(W_2,B_2),\ldots,(W_H,B_H),(W_{H+1},B_{H+1})\right).
\end{align}
First, 
the definitions of $\funcCalL$ and $\funcCalR$ (see \eqref{m07c} and \cref{m07d}),
\eqref{b08a}, and the fact that $\forall \, i\in [1,M]\cap\N\colon f_i\in C(\R^p,\R^q)$ show for all $i\in [1,M]\cap\N$ that $
k_i=
(p,k_{i,1},k_{i,2},\ldots,k_{i,H},q).
$
The definition of $\funcCalL$ (see \eqref{m07c}),
the definition of $\dimsum$ (see \eqref{b26c}),
and \eqref{b01b} then show that
\begin{align}\funcCalL(\psi)= (p,k_1^{\dimsum},\ldots,k_H^{\dimsum},q)=\dimsum_{i=1}^Mk_i.\label{b23}
\end{align}
Next, we prove by induction on $n\in [1,H]\cap\N$ that $ x_n=(x_{1,n},x_{2,n},\ldots,x_{M,n})$.
First, \eqref{b01c} shows  that
\begin{align}
W_1x_0+B_1= 
\begin{pmatrix}
W_{1,1}\\
W_{2,1}\\
\vdots\\
W_{M,1}
\end{pmatrix}x_0+
\begin{pmatrix}
B_{1,1}\\
B_{2,1}\\
\vdots\\
B_{M,1}
\end{pmatrix}
=
\begin{pmatrix}
W_{1,1}x_0+B_{1,1}\\
W_{2,1}x_0+B_{2,1}\\
\vdots\\
W_{M,1}x_0+B_{M,1}
\end{pmatrix}.\label{b20}
\end{align}
This implies that
\begin{align}
x_1= \funcBoldA{k_1^{\dimsum}}(W_1x_0+B_1)=\begin{pmatrix}
x_{1,1}\\x_{2,1}\\\vdots\\x_{M,1}\end{pmatrix}.
\end{align}
This proves the base case. Next, for the induction step let $n\in [2,H]\cap\N$ and assume that $x_{n-1}=(x_{1,n-1},x_{2,n-1},\ldots,x_{M,n-1})$.
Then \eqref{b01d} and the induction hypothesis ensure that
\begin{align}\begin{split}
&W_nx_{n-1}+B_n
\\
&= W_{n}\begin{pmatrix}
x_{1,n-1}\\
x_{2,n-1}\\
\vdots\\
x_{M,n-1}
\end{pmatrix}+B_{n}
=\begin{pmatrix}
W_{1,n} &	0		&	0		&	0	\\
0 		&	W_{2,n}	&	0		&	0	\\
0		&	0		&	\ddots	&	0	\\
0		&	0		&	0		&W_{M,n}
\end{pmatrix}
\begin{pmatrix}
x_{1,n-1}\\
x_{2,n-1}\\
\vdots\\
x_{M,n-1}
\end{pmatrix}+
\begin{pmatrix}
B_{1,n}\\
B_{2,n}\\
\vdots\\
B_{M,n}
\end{pmatrix}
\\
&=
\begin{pmatrix}
W_{1,n}x_{1,n-1}+
B_{1,n}\\
W_{2,n}x_{2,n-1}+B_{2,n}\\
\vdots\\
W_{M,n}x_{M,n-1}+ B_{M,n}
\end{pmatrix}.\label{b21}\end{split}
\end{align}
This yields that
\begin{align}
x_{n}= \funcBoldA{k_n^{\dimsum}}(W_nx_{n-1}+B_n)=\begin{pmatrix}
x_{1,n}\\x_{2,n}\\\vdots\\x_{M,n}
\end{pmatrix}.
\end{align}
This proves the induction step. Induction now proves for all $n\in [1,H]\cap\N$ that
$x_n=(x_{1,n},x_{2,n},\ldots,x_{M,n})$.
This, the definition of $\funcCalR$ (see \eqref{m07d}), and
\eqref{b01e} imply that
\begin{align}\begin{split}
&(\funcCalR(\psi))(x_0)=W_{H+1}x_H+B_{H+1}\\
&=W_{H+1}\begin{pmatrix}
x_{1,H}\\
x_{2,H}\\
\vdots\\
x_{M,H}
\end{pmatrix}+B_{H+1}
=\begin{pmatrix}
h_1W_{1,H+1}&\ldots&h_MW_{M,H+1}
\end{pmatrix}
\begin{pmatrix}
x_{1,H}\\
x_{2,H}\\
\vdots\\
x_{M,H}
\end{pmatrix}+\left[\sum_{i=1}^{M}h_iB_{i,H+1}\right]\\
&=\left[\sum_{i=1}^{M}h_iW_{i,H+1}x_{i,H}\right]+\left[\sum_{i=1}^{M}h_iB_{i,H+1}\right]=
\sum_{i=1}^{M}h_i\left(W_{i,H+1}x_{i,H}+B_{i,H+1}\right)\\&=\sum_{i=1}^M h_i(\funcCalR(\phi_i))(x_0).
\end{split}
\end{align}
This, the fact that $x_0\in \R^{p}$  was arbitrary, and \eqref{b08a} yield that
\begin{align}
\funcCalR(\psi)= \sum_{i=1}^{M}h_i\funcCalR(\phi_i)=\sum_{i=1}^{M}h_if_i.
\end{align}
This and \eqref{b23} show that
\begin{align}
\sum_{i=1}^{M}h_if_i
\in\funcCalR\left(\left\{ \Phi\in\setCalN\colon
\funcCalL(\Phi)=\dimsum_{i=1}^Mk_i\right\}\right).
\end{align}
The proof of \cref{b01} is thus completed.
\end{proof}

\subsection{Deep neural network representations for MLP approximations}
\begin{lemma}\label{b26}
Assume \cref{m07}, 
let $d,M\in \N$, 
$T,c \in (0,\infty)$, 
$f\in C(\R,\R)$, 
$g \in C( \R^d, \R)$,
$\Phi_f,\Phi_\funcG\in \setCalN$ satisfy that
$\funcCalR(\Phi_f)= f$,
$\funcCalR(\Phi_\funcG)= \funcG$,
and
\begin{align}\label{b12}
c\geq\max\left\{2, \supnorm{\funcCalL(\funcPhi{f})},\supnorm{\funcCalL(\funcPhi{\funcG})}\right\},
\end{align}
let
$(\Omega, \mathcal{F}, \P)$
be a probability space, 
let
$  \Theta = \bigcup_{ n \in \N } \Z^n$,
let $\unif^\theta\colon \Omega\to[0,1]$, $\theta\in \Theta$, be independent random variables which are uniformly distributed on $[0,1]$, 
let $\uniform^\theta\colon [0,T]\times \Omega\to [0, T]$, $\theta\in\Theta$, satisfy 
for all $t\in [0,T]$, $\theta\in \Theta$ that 
$\uniform^\theta _t = t+ (T-t)\unif^\theta$,
let $\sppr^\theta\colon[0,T]\times \Omega \to\R^d $, $\theta\in \Theta$, be independent 
%$d$-dimensional 
standard Brownian motions with continuous sample paths,
assume that $(\unif^\theta)_{\theta\in \Theta}$ and $(\sppr^\theta)_{\theta\in \Theta}$ are independent,
let
$ 
  {\bigU}_{ n,M}^{\theta } \colon [0, T] \times \R^d \times \Omega \to \R
$, $n,M\in\Z$, $\theta\in\Theta$, satisfy
for all $n \in \N$, $\theta\in\Theta $, 
$ t \in [0,T]$, $x\in\R^d $
that ${\bigU}_{-1,M}^{\theta}(t,x)={\bigU}_{0,M}^{\theta}(t,x)=0$ and 
%\begin{equation}  \begin{split}
\begin{equation}
\begin{split}
&  {\bigU}_{n,M}^{\theta}(t,x)
=
  \frac{1}{M^n}
 \sum_{i=1}^{M^n} 
      \funcG\big(x+\sppr^{(\theta,0,-i)}_{T}-\sppr^{(\theta,0,-i)}_{t}\big)
 \\
 &+
  \sum_{l=0}^{n-1} \frac{(T-t)}{M^{n-l}}
    \left[\sum_{i=1}^{M^{n-l}}
      \big(f\circ{\bigU}_{l,M}^{(\theta,l,i)}-\1_{\N}(l)f\circ {\bigU}_{l-1,M}^{(\theta,-l,i)}\big)\!\!
      \left(\uniform_t^{(\theta,l,i)},x+\sppr_{\uniform_t^{(\theta,l,i)}}^{(\theta,l,i)}-\sppr_{t}^{(\theta,l,i)}\right)
    \right],
\end{split}\label{t27b}
\end{equation}
and let $\omega\in\Omega$.
 Then 
 for all $n\in \N_0$
there exists  a family $(\netflow{n,t}{\theta})_{\theta\in\Theta,t\in [0,T]}\subseteq \setCalN$ such that 
\begin{enumerate}[(i)]
\item \label{it1} it holds for all
% $n\in\N_0$
$t_1,t_2\in [0,T]$, $\theta_1,\theta_2\in\Theta$
that  
 \begin{align}\label{b13}
\funcCalL\left(\netflow{n,t_1}{\theta_1}\right)=\funcCalL\left(\netflow{n,t_2}{\theta_2}\right),
\end{align}
\item it holds for all 
%$n \in \N_0$, 
$t\in [0,T]$, $ \theta\in\Theta$
 that
\begin{align}\label{b14}\begin{split}
&\cardL{\netflow{n,t}{\theta}}= n\Big(\!\cardL{\funcPhi{f}}-1\Big)+\cardL{\funcPhi{\funcG}},
\end{split}
\end{align}
\item it holds for all 
$t\in [0,T]$, $ \theta\in\Theta$
 that
\begin{align}\label{b14a}\begin{split}
&\supnorm{\funcCalL(\netflow{n,t}{\theta} )}\leq c(3 M)^n,
\end{split}
\end{align}
and
\item \label{it3}
it holds for all
 $\theta\in \Theta$, $t\in[0,T]$, $x\in \R^d$  that
\begin{align}
{\bigU}_{n,M}^{\theta}(t,x,\omega)=(\funcCalR(\netflow{n,t}{\theta}))(x).\label{b14b}
\end{align}
\end{enumerate}
\end{lemma}
\begin{proof}[Proof of \cref{b26}]
%We will prove the result in two steps: First, we  prove
%by induction on $n\in \N_0$
%that there exists a family of DNNs $(\netflow{n,t}{\theta})_{\theta\in\Theta,t\in [0,T],n\in \N_0}\subset \setCalN$ which satisfies
%\eqref{b12},
%\eqref{b13},
%\eqref{b14}, and
%\eqref{b14b}.  Second, we show that this family also satisfies \eqref{b14c}. 
We prove \cref{b26} by induction on $n\in \N_0$. 
For the base case $n=0$ note that
the fact that 
$\forall \, t\in [0,T],\theta\in \Theta\colon U^\theta_{0,M}(t,\cdot)=0 $,
the fact that the function $0$ can be represented by a network with depth $\cardL{\funcPhi{\funcG}}$, and \eqref{b12}
 imply that there exists 
$(\netflow{0,t}{\theta})_{\theta\in \Theta, t\in[0,T]}\subseteq\setCalN$ such that
 it holds for all
$
t_1,t_2\in [0,T]$, $\theta_1,\theta_2\in\Theta$
that
$
 \funcCalL\left(\netflow{0,t_1}{\theta_1}\right)=\funcCalL\left(\netflow{0,t_2}{\theta_2}\right)
$
and such that
 it holds for all $ \theta\in \Theta$, $t\in[0,T]$ that
$
 \card{\funcCalL(\netflow{0,t}{\theta})}=\cardL{\funcPhi{\funcG}}$, $
\supnorm{\funcCalL(\netflow{0,t}{\theta} )}\leq\supnorm{\funcCalL(\funcPhi{g})}\leq c$, and $
   {\bigU}_{0,M}^{\theta}(t,\cdot,\omega)= \funcCalR(\netflow{0,t}{\theta})$.
This proves
the base case $n=0$. 
%
%the following induction hypothesis: let $n\in \N_0$,  let
%$(\netflow{k,t}{\theta})_{t\in [0,T],\theta\in\Theta,k\in [0,n]\cap\N_0}\subset \setCalN$,
%assume  
%for all $k\in [0,n]\cap \N_0$
%$t_1,t_2\in [0,T]$, $\theta_1,\theta_2\in\Theta$
%that  
%\begin{align}\funcCalL\left(\netflow{n,t_1}{\theta_1}\right)=\funcCalL\left(\netflow{n,t_2}{\theta_2}\right)\label{b20}
%\end{align}
%assume  for all $ k\in[0,n] \cap \N$, $t\in[0,T]$, $\theta\in\Theta$, $\alpha \in [5,\infty)$ that
%\begin{align}
% &\cardL{\netflow{k,t}{\theta}}=k(\cardL{\funcPhi{f}}+1)+\cardL{\funcPhi{\funcG}}+3\label{b24}\quad\text{and}\quad
%\\
%& \supnorm{\funcCalL(\netflow{k,t}{\theta} )}\leq c(\alpha M)^k,
% \label{b21}
%\end{align}
%and
%assume
%for all $k\in[0,n]\cap \N_0$, $\theta\in \Theta$, $t\in[0,T]$, $\omega\in\Omega$  that
%\begin{align}
%{\bigU}_{n,M}^{\theta}(t,\cdot,\omega)=\funcCalR(\netflow{n,t}{\theta}).
%\end{align}

For the induction step from $n\in\N_0$ to $n+1\in\N$ let $n\in \N_0$ and assume 
that \cref{it1}--\cref{it3} hold true for all $k\in [0,n]\cap \N_0$. The assumption that 
$\funcG=\funcCalR(\Phi_\funcG)$ and
\cref{p01} (with $d=d$, $m=1$, $\lambda=1$, $a=0$, $b=\sppr^{\theta}_{T}(\omega)-\sppr^{\theta}_{t}(\omega)$, and $\Psi=\Phi_\funcG$ 
for $\theta\in \Theta$, $t\in [0,T]$
in the notation of \cref{p01})
show for all $\theta\in \Theta$, $t\in [0,T]$ that
\begin{align}
\begin{split}
      \funcG\big(\cdot+\sppr^{\theta}_{T}(\omega)-\sppr^{\theta}_{t}(\omega)\big)
&=(\funcCalR(\Phi_\funcG))\big(\cdot+\sppr^{\theta}_{T}(\omega)-\sppr^{\theta}_{t}(\omega)\big)
\\
&\in \funcCalR\left(\left\{\Phi\in\setCalN\colon
\funcCalL(\Phi)=\funcCalL(\funcPhi{\funcG})\big.
\right\}\Big.\!\right).
\end{split}\label{b09b}
\end{align}
Furthermore, \cref{b03} (with
$H=(n+1)\left(\card{\funcCalL(\funcPhi{f})}-1\right)-1$ in the notation of \cref{b03})
ensures that
\begin{align}
\identity{\R}\in \funcCalR\left(\left\{\Phi\in\setCalN\colon \funcCalL(\Phi)=\neutralDim{(n+1)\left(\card{\funcCalL(\funcPhi{f})}-1\right)+1}\Big.\right\}\bigg.\!\right).
\end{align}
This, \eqref{b09b}, 
and \cref{m11b} (with $d_1=d$, $d_2=1$, $d_3=1$, $f=\identity{\R}$, $g=\funcG\big(\cdot+\sppr^{\theta}_{T}(\omega)-\sppr^{\theta}_{t}(\omega)\big)$, $\alpha=\neutralDim{(n+1)\left(\card{\funcCalL(\funcPhi{f})}-1\right)+1}$, and $\beta=\funcCalL(\Phi_\funcG)$ 
for  $\theta\in \Theta$, $t\in [0,T]$
in the notation of \cref{m11b}) show that 
for all $\theta\in \Theta$, $t\in [0,T]$ it holds that
\begin{align}\begin{split}
\funcG\big(\cdot+\sppr^{\theta}_{T}(\omega)-\sppr^{\theta}_{t}(\omega)\big)\in  \funcCalR\bigg(\bigg\{\Phi\in\setCalN\colon
\funcCalL(\Phi)=\neutralDim{(n+1)\left(\card{\funcCalL(\funcPhi{f})}-1\right)+1} \concat\funcCalL(\funcPhi{\funcG})
\bigg\}\bigg).\label{b09}
\end{split}
\end{align}
Next, the induction hypothesis implies
for all $\theta\in \Theta$, $t\in[0,T]$,
$ l\in [0,n]\cap\N_0$
 that
\begin{align}
{\bigU}_{l,M}^{\theta}(t,\cdot,\omega)=\funcCalR(\netflow{l,t}{\theta})\quad\text{and}\quad \funcCalL\left(\netflow{l,t}{\theta}\right)=\funcCalL\left(\netflow{l,0}{0}\right). 
\end{align}
This and 
\cref{p01} (with 
\begin{align}\begin{split}
&d=d,\quad m=1,\quad a=0,\quad b=\sppr_{\uniform_t^{\theta}(\omega)}^{\theta}(\omega)-
\sppr_{t}^{\theta}(\omega),\quad\text{and}\\ &\Psi=\netflow{l,\uniform_t^{\theta}(\omega)}{\eta}\quad 
\text{for}\quad  \theta,\eta\in \Theta, \quad t\in [0,T],\quad l\in [0,n]\cap\N_0 \end{split}
\end{align}
in the notation of \cref{p01}) imply that for all $\theta,\eta\in \Theta$, $t\in [0,T]$,  $l\in [0,n]\cap\N_0$ it holds that 
\begin{align}\begin{split}
&
 U_{l,M}^{\eta} \left(\uniform_t^{\theta}(\omega),\cdot+\sppr_{\uniform_t^{\theta}(\omega)}^{\theta}(\omega)-
\sppr_{t}^{\theta}(\omega),\omega
\right)\\
&=\left(\funcCalR\big(\netflow{l,\uniform_t^{\theta}(\omega)}{\eta}\big)\right)\!\!\left(\cdot+\sppr_{\uniform_t^{\theta}(\omega)}^{\theta}(\omega)-
\sppr_{t}^{\theta}(\omega)\right)
\\
&\in   \funcCalR\bigg(\bigg\{\Phi\in\setCalN\colon
\funcCalL(\Phi)= \funcCalL\left(\netflow{l,\uniform_t^{\theta}(\omega)}{\eta}\right)
\bigg\}\bigg)=   \funcCalR\bigg(\bigg\{\Phi\in\setCalN\colon
\funcCalL(\Phi)= \funcCalL\left(\netflow{l,0}{0}\right)
\bigg\}\bigg).
\end{split}\label{b10b}
\end{align}
Moreover, \cref{b03} (with
$H=(n-l)\left(\card{\funcCalL(\funcPhi{f})}-1\right) -1$ for 
$l\in [0,n-1]\cap\N_0$
 in the notation of \cref{b03})
ensures for all $l\in [0,n-1]\cap\N_0$ that
\begin{align}
\identity{\R}\in \funcCalR\left(\left\{\Phi\in\setCalN\colon \funcCalL(\Phi)=\neutralDim{(n-l)\left(\card{\funcCalL(\funcPhi{f})}-1\right) +1} \Big.\right\}\bigg.\!\right).
\end{align}
This, \eqref{b10b}, and \cref{m11b} (with 
\begin{align}\begin{split}
&d_1=d, \quad d_2=1, \quad d_3=1, \quad f=\identity{\R}, 
\quad \alpha=\neutralDim{(n-l)\left(\card{\funcCalL(\funcPhi{f})}-1\right) +1}, \quad\\
&\beta=\funcCalL\left(\netflow{l,0}{0}\right),\quad\text{and}\quad 
g= U_{l,M}^{\eta} \left(\uniform_t^{\theta}(\omega),\cdot+\sppr_{\uniform_t^{\theta}(\omega)}^{\theta}(\omega)-
\sppr_{t}^{\theta}(\omega),\omega
\right)\\
&\qquad\qquad
\text{for}\quad\eta,\theta\in \Theta,\quad t\in [0,T],\quad l\in [0,n-1]\cap\N_0\\
\end{split}
\end{align}
in the notation of \cref{m11b}) prove for all $\eta,\theta\in \Theta$, $t\in [0,T]$, $l\in [0,n-1]\cap\N_0$ that
\begin{align}
\begin{split}
& U_{l,M}^{\eta} \left(\uniform_t^{\theta}(\omega),\cdot+\sppr_{\uniform_t^{\theta}(\omega)}^{\theta}(\omega)-
\sppr_{t}^{\theta}(\omega),\omega
\right)
\\
&\in    \funcCalR\bigg(\bigg\{\Phi\in\setCalN\colon
\funcCalL(\Phi)= \neutralDim{(n-l)\left(\card{\funcCalL(\funcPhi{f})}-1\right) +1} 
\concat \funcCalL(\netflow{l,0}{0})
\bigg\}\bigg).
\end{split}\label{m10c}
\end{align}
This and \cref{m11b} (with 
\begin{align}\begin{split}
&d_1=d, \quad d_2=1, \quad d_3=1, \quad f=f,\quad  \alpha= \funcCalL(\Phi_f),\quad\\ &\beta=\neutralDim{(n-l)\left(\card{\funcCalL(\funcPhi{f})}-1\right) +1} 
\concat \funcCalL(\netflow{l,0}{0}),\quad\text{and}
\quad g=U_{l,M}^{\eta} \left(\uniform_t^{\theta}(\omega),\cdot+\sppr_{\uniform_t^{\theta}(\omega)}^{\theta}(\omega)-
\sppr_{t}^{\theta}(\omega),\omega
\right)\\
&\qquad\qquad\qquad\text{for}\quad \eta,\theta\in \Theta,\quad t\in [0,T], \quad l\in [0,n-1]\cap\N_0
\end{split}
\end{align}
in the notation of \cref{m11b}) assure for all $\eta,\theta\in \Theta$, $t\in [0,T]$, $l\in [0,n-1]\cap\N_0$ that
\begin{align}\begin{split}
&      \left(f\circ U_{l,M}^{\eta}\right) \left(\uniform_t^{\theta}(\omega),\cdot+\sppr_{\uniform_t^{\theta}(\omega)}^{\theta}(\omega)-
\sppr_{t}^{\theta}(\omega),\omega
\right) \\
&\in   \funcCalR\bigg(\bigg\{\Phi\in\setCalN\colon
\funcCalL(\Phi)=\funcCalL(\funcPhi{f})\concat \neutralDim{(n-l)\left(\card{\funcCalL(\funcPhi{f})}-1\right) +1} 
\concat \funcCalL(\netflow{l,0}{0})
\bigg\}\bigg).
\end{split}\label{b10}
\end{align}
Next, \eqref{b10b} (with $l=n$) and 
\cref{m11b} (with 
\begin{align}\begin{split}
&d_1=d, \quad d_2=1, \quad d_3=1, \quad
f=f,\quad 
\alpha=  \funcCalL(\Phi_f),\quad
\beta=\funcCalL\left(\netflow{n,0}{0}\right),\quad\text{and}\\
&
g=\left( U_{n,M}^{\eta}\right)
\left(\uniform_t^{\theta}(\omega),\cdot+\sppr_{\uniform_t^{\theta}(\omega)}^{\theta}(\omega)-
\sppr_{t}^{\theta}(\omega),\omega
\right)\quad \text{for}\quad
\eta,\theta\in \Theta,\quad t\in [0,T]
\end{split}
\end{align}
in the notation of \cref{m11b}) prove
for all $\eta,\theta\in \Theta$, $t\in [0,T]$ that
\begin{align}
\begin{split}
&    \left(f\circ U_{n,M}^{\eta}\right) \left(\uniform_t^{\theta}(\omega),\cdot+\sppr_{\uniform_t^{\theta}(\omega)}^{\theta}(\omega)-
\sppr_{t}^{\theta}(\omega),\omega
\right) \\
&\in   \funcCalR\bigg(\bigg\{\Phi\in\setCalN\colon
\funcCalL(\Phi)=\funcCalL(\funcPhi{f})
\concat \funcCalL(\netflow{n,0}{0})
\bigg\}\bigg).
\end{split}\label{b10c}
\end{align}
Furthermore, the definition of $\concat$ in \eqref{b25} and 
the fact that
\begin{align}
\forall \,
l\in [0,n]\cap\N_0\colon 
\card{\funcCalL(\netflow{l,0}{0})}=l \left(\cardL{\funcPhi{f}}-1\right)+
\cardL{\funcPhi{\funcG}}\end{align} in 
the induction hypothesis imply that
\begin{align}
\begin{split}
&\card{\neutralDim{(n+1)\left(\card{\funcCalL(\funcPhi{f})}-1\right)+1}\concat \funcCalL(\funcPhi{\funcG})}\\
&=\Big[(n+1)\Big(\card{\funcCalL(\funcPhi{f})}-1\Big)+1\Big]+\cardL{\funcPhi{\funcG}}-1\\
&=(n+1)\Big(\card{\funcCalL(\funcPhi{f})}-1\Big)+\cardL{\funcPhi{\funcG}},
\end{split}
\end{align}
that
\begin{align}\begin{split}
&
\card{ 
\funcCalL(\funcPhi{f})
  \concat \funcCalL (\netflow{n,0}{0}) }
=
\card{ 
\funcCalL(\funcPhi{f})}+\card{ \funcCalL (\netflow{n,0}{0})}
-1\\
&=
\card{ 
\funcCalL(\funcPhi{f})}+\Big[n\Big(\!\cardL{\funcPhi{f}}-1\Big)+\cardL{\funcPhi{\funcG}}\Big]
-1\\
&=
(n+1)\Big(\!\card{\funcCalL(\funcPhi{f})}-1\Big)+\cardL{\funcPhi{\funcG}},\end{split}
\end{align}
and for all 
$l\in [0,n-1]\cap\N_0$ that
\begin{align}\begin{split}
&
\card{ 
\funcCalL(\funcPhi{f})
\concat\neutralDim{(n-l)\left(\card{\funcCalL(\funcPhi{f})}-1\right) +1}  \concat \funcCalL (\netflow{l,0}{0}) 
}\\
&= \card{ \funcCalL(\funcPhi{f})}+\card{\neutralDim{(n-l)\left(\card{\funcCalL(\funcPhi{f})}-1\right) +1}}+\card{ \funcCalL (\netflow{l,0}{0}) }-2\\
&=\card{ \funcCalL(\funcPhi{f})}+\Big[(n-l)\left(\Big.\!\card{\funcCalL(\funcPhi{f})}-1\right)+1 \Big]  \\&\qquad\qquad+
\Big[l \left(\cardL{\funcPhi{f}}-1\Big.\!\right)+
\cardL{\funcPhi{\funcG}}\Big]-2\\
&=\card{ \funcCalL(\funcPhi{f})}+ n\left(\Big.\!\card{\funcCalL(\funcPhi{f})}-1\right)+\cardL{\funcPhi{\funcG}}-1\\
&= (n+1)\left(\Big.\!\card{\funcCalL(\funcPhi{f})}-1\right)+\cardL{\funcPhi{\funcG}}.
\end{split}\label{b11}\end{align} 
This shows, roughly speaking, that the functions in \eqref{b09},
\eqref{b10c}, and \eqref{b10} can be represented by networks with the same depth
(i.e.\ number of layers): $(n+1)(\card{\funcCalL(\funcPhi{f})}-1)+\cardL{\funcPhi{\funcG}}$. Hence, \cref{b01} 
\begin{comment}
(with
\begin{align}TODO
H=(n+1)(\card{\funcCalL(\funcPhi{f})}-1)+\cardL{\funcPhi{\funcG}}-2,\quad  k_0=d,\quad  k_{H+1}=1,  
\end{align}
and $f_i$, $i\in [1,M]\cap\N $, being the functions
\begin{align}TODO
\funcG\big(\cdot+\sppr^{(\theta,0,-i)}_{T}(\omega)-\sppr^{(\theta,0,-i)}_{t}(\omega),\quad i\in [1,M^{n+1}]\cap\N,\\
\left(f\circ{\bigU}_{l,M}^{(\theta,l,i)}-\1_{\N}(l)f\circ {\bigU}_{l-1,M}^{(\theta,-l,i)}\right)
      \left(\uniform_t^{(\theta,l,i)}(\omega),\cdot+\sppr_{\uniform_t^{(\theta,l,i)}}^{(\theta,l,i)}(\omega)-
\sppr_{t}^{(\theta,l,i)}(\omega),\omega
\right)
\end{align}
in the notation of \cref{b01})\end{comment}
and \eqref{t27b} imply that
there exists a family
$(\netflow{n+1,t}{\theta})_{\theta\in \Theta, t\in [0,T]}\subseteq\setCalN$  such that  for all $\theta\in \Theta$, $t\in [0,T]$, $x\in\R^d$  it holds that
\begin{equation}
\begin{split}
&\left(\funcCalR(\netflow{n+1,t}{\theta})\right)(x) \\
&=
  \frac{1}{M^{n+1}}
 \sum_{i=1}^{M^{n+1}} 
      \funcG\big(x+\sppr^{(\theta,0,-i)}_{T}(\omega)-\sppr^{(\theta,0,-i)}_{t}(\omega)\big)
 \\
 &\qquad+
   \frac{(T-t)}{M}
    \sum_{i=1}^{M}
      \left(f\circ{\bigU}_{n,M}^{(\theta,n,i)}\right)\left(\uniform_t^{(\theta,n,i)}(\omega),x+\sppr_{\uniform_t^{(\theta,n,i)}(\omega)}^{(\theta,n,i)}(\omega)-
\sppr_{t}^{(\theta,n,i)}(\omega),\omega
\right)\\
 &\qquad+
  \sum_{l=0}^{n-1} \frac{(T-t)}{M^{n+1-l}}
    \sum_{i=1}^{M^{n+1-l}}
      \left(f\circ{\bigU}_{l,M}^{(\theta,l,i)}\right)\left(\uniform_t^{(\theta,l,i)}(\omega),x+\sppr_{\uniform_t^{(\theta,l,i)}(\omega)}^{(\theta,l,i)}(\omega)-
\sppr_{t}^{(\theta,l,i)}(\omega),\omega
\right)\\
&\qquad-\sum_{l=1}^{n} \frac{(T-t)}{M^{n+1-l}}
    \sum_{i=1}^{M^{n+1-l}}
\left(f\circ {\bigU}_{l-1,M}^{(\theta,-l,i)}  \right) \left(\uniform_t^{(\theta,l,i)}(\omega),x+\sppr_{\uniform_t^{(\theta,l,i)}(\omega)}^{(\theta,l,i)}(\omega)-
\sppr_{t}^{(\theta,l,i)}(\omega),\omega
\right)
   \\
&= {\bigU}_{n+1,M}^{\theta}(t,x,\omega),
\end{split}\label{b16}
\end{equation}
that
\begin{align}\label{b19}
\card{
\funcCalL(\netflow{n+1,t}{\theta})}=
(n+1)(\card{\funcCalL(\funcPhi{f})}-1)+\cardL{\funcPhi{\funcG}},
\end{align}
and that
\begin{align}\begin{split}
\funcCalL(\netflow{n+1,t}{\theta})
&=
\left(
\dimsum_{i=1}^{M^{n+1}}\left[\neutralDim{(n+1)\left(\card{\funcCalL(\funcPhi{f})}-1\right)+1} \concat\funcCalL(\funcPhi{\funcG})\Big.\right]
\right)
 \dimsum\Bigg( \dimsum_{i=1}^{M}
\left(\funcCalL\left(\funcPhi{f}\right)
\concat \funcCalL\left(\netflow{n,0}{0}\right)\right)\Bigg)
\\%%2
&\quad \dimsum \Bigg(\dimsum_{l=0}^{n-1}\dimsum_{i=1}^{M^{n+1-l}}\bigg[
\left(\funcCalL(\funcPhi{f})\concat \neutralDim{(n-l)(\card{\funcCalL(\funcPhi{f})}-1) +1 }
\concat \funcCalL(\netflow{l,0}{0})\right)\Bigg)
\\
&\quad \dimsum\Bigg(
\dimsum_{l=1}^{n}\dimsum_{i=1}^{M^{n+1-l}}
\left(\funcCalL(\funcPhi{f})\concat \neutralDim{(n-l+1)(\card{\funcCalL(\funcPhi{f})}-1) +1 }
\concat \funcCalL(\netflow{l-1,0}{0})\right)\bigg]\Bigg)%%4
.\end{split}\label{b17}
\end{align}
This shows for all $t_1,t_2\in [0,T]$, $\theta_1,\theta_2\in\Theta$
that  
\begin{align}
\funcCalL\left(\netflow{n+1,t_1}{\theta_1}\right)=\funcCalL\left(\netflow{n+1,t_2}{\theta_2}\right).\label{b22}
\end{align}
Furthermore, \eqref{b17}, the triangle inequality (see \cref{b15}), and
the fact that
\begin{align}
\forall \, l\in [0,n]\cap\N_0\colon\supnorm{\funcCalL(\netflow{l,0}{0} )}\leq c(3 M)^l\label{b22b}
\end{align}
 in
the induction hypothesis show
 for all $\theta\in \Theta$, $t\in [0,T]$ that
\begin{align}\begin{split}\label{b22c}
\supnorm{\funcCalL(\netflow{n+1,t}{\theta})}&\leq 
\sum_{i=1}^{M^{n+1}}\supnorm{ \neutralDim{(n+1)(\card{\funcCalL(\funcPhi{f})}-1)+1} \concat\funcCalL(\funcPhi{\funcG})}+\sum_{i=1}^{M}\supnorm{\funcCalL(\funcPhi{f})
\concat \funcCalL(\netflow{n,0}{0})}\\
&\quad+ \sum_{l=0}^{n-1}\sum_{i=1}^{M^{n+1-l}}
\supnorm{\funcCalL(\funcPhi{f})\concat \neutralDim{(n-l)(\card{\funcCalL(\funcPhi{f})}-1) +1 }
\concat \funcCalL(\netflow{l,0}{0})}\\
&\quad+ \sum_{l=1}^{n}\sum_{i=1}^{M^{n+1-l}}\supnorm{\funcCalL(\funcPhi{f})\concat \neutralDim{(n-l+1)(\card{\funcCalL(\funcPhi{f})}-1) +1 }
\concat \funcCalL(\netflow{l-1,0}{0})}.
\end{split}
\end{align} 
Note that for all $H_1,H_2,\alpha_0,\ldots,\alpha_{H_1+1},\beta_0,\ldots, \beta_{H_2+1}\in \N$, $\alpha,\beta \in \setCalD$ with $\alpha=(\alpha_0,\ldots,\alpha_{H_1+1})$, $\beta=(\beta_0,\ldots, \beta_{H_2+1})$, $\alpha_0=\beta_{H_2+1}=1$ it holds that $\supnorm{\alpha\concat\beta}\leq \max\{\supnorm{\alpha},\supnorm{\beta},2\}$ (see~\eqref{b25}). This, \cref{b22c}, 
the fact that
$\forall \, H\in \N\colon \supnorm{\neutralDim{H+2}}=2$ (see \eqref{b26d}),
\eqref{b12}, and \eqref{b22b} prove that
\begin{align}
\begin{split}
&\supnorm{\funcCalL(\netflow{n+1,t}{\theta})}\\
&\leq \left[
\sum_{i=1}^{M^{n+1}}c \right]+
\left[ \sum_{i=1}^{M}c({3} M)^n\right]+
\left[ \sum_{l=0}^{n-1}\sum_{i=1}^{M^{n+1-l}}c({3} M)^l\right]
+\left[ \sum_{l=1}^{n}\sum_{i=1}^{M^{n+1-l}}c({3} M)^{l-1}\right]
\\
&= 
M^{n+1}c+Mc(3M)^{n}+\left[\sum_{l=0}^{n-1}M^{n+1-l}c(3M)^l\right]+\left[
\sum_{l=1}^{n}M^{n+1-l}c(3M)^{l-1}\right]
\\
&= 
M^{n+1}c\left[1+3^n+\sum_{l=0}^{n-1}3^l+\sum_{l=1}^{n}3^{l-1}\right]=
M^{n+1}c\left[1+\sum_{l=0}^{n}3^l+\sum_{l=1}^{n}3^{l-1}\right]
\\
&\leq 
 cM^{n+1}\left[1+2\sum_{l=0}^{n} {3} ^l\right]= cM^{n+1}\left[1+2\frac{{3}^{n+1}-1}{{3}-1}\right]
= c({3} M)^{n+1}.
\end{split}
\label{b18}
\end{align}
Combining \eqref{b16}, \eqref{b19}, \eqref{b22}, and \eqref{b18} completes the induction step. Induction hence establishes \cref{it1}--\cref{it3}. 
The proof of \cref{b26} is thus completed.
\end{proof}

\subsection{Deep neural network approximations for the PDE nonlinearity}
\begin{lemma}[DNN interpolation]\label{xn11}Assume \cref{m07},
let $N\in\N$, $a_0,a_1,\ldots, a_{N-1},\gp_0,\gp_1,\ldots,\gp_N\in \R$ satisfy that $\gp_0<\gp_1<\ldots<\gp_N$,
let
%$ a_n,t_n\in\R$, $n\in [0,N]\cap\Z$, 
%satisfy for all $n\in [0,N-1]\cap\Z$  that 
%$  t_n<t_{n+1}$ and $
%a_N=0$
% and
$f\colon \R\to\R$
be a function,
assume for all
$x\in (-\infty,\gp_0]$ that $f(x)=f(\gp_0)$, 
 assume for all
$n\in [0,N-1]\cap \Z$,
$x\in(\gp_n,\gp_{n+1}]$ that $f(x)=f(\gp_n)+a_n(x-\gp_n)$, and
assume for all
$x\in (\gp_N,\infty)$ that $f(x)=f(\gp_N)$.
Then it holds that
\begin{align}
f\in \funcCalR(\{\Phi \in\setCalN\colon \funcCalL(\Phi)=(1,N+1,1)\}).\label{x06}
\end{align}
\end{lemma}
\begin{proof}[Proof of \cref{xn11}]
Throughout this proof let $a_{-1}=0$ and $a_N=0$,
let $c_n$, $n\in [0,N]\cap\Z$, be the real numbers which satisfy 
for all
$
 n\in [0,N]\cap\Z$ that $ c_n=a_{n}-a_{n-1}$,
let 
$W_1\in \R^{(N+1)\times 1}$, $B_1\in \R ^{N+1}$, $W_2\in\R^{1\times(N+1)}$, $B_2\in \R$, $\Phi\in\setCalD$  be given by
\begin{align}
&W_1 = \begin{pmatrix}
1\\1\\\vdots\\1
\end{pmatrix} ,\quad 
B_1=\begin{pmatrix}
-\gp_0\\-\gp_1\\\vdots\\-\gp_N
\end{pmatrix} ,\quad 
W_2= \begin{pmatrix}
c_0&c_1&\ldots&c_N
\end{pmatrix},\quad B_2= f(\gp_0)\label{x01}
\end{align}
and
\begin{align}
\Phi= ((W_1,B_1),(W_2,B_2)),\label{x02}
\end{align}
and let $g\colon \R\to\R$ satisfy for all $x\in\R$ that
\begin{align}
g(x)=f(\gp_0)+\sum_{k=0}^{N}c_k\max\{x-\gp_k,0\}.\label{x04}
\end{align}
First, observe that the fact that
$\forall \,n\in [0,N-1]\cap\Z\colon   \gp_n<\gp_{n+1}$ and the fact that
$\forall\,
 n\in [0,N]\cap\Z\colon   a_n= \sum_{k=0}^{n}c_k$
then show for all
$n\in [0,N-1]\cap\Z$, $x\in (\gp_n,\gp_{n+1}]$ that
\begin{align}\begin{split}
g(x)-g(\gp_n)&= \left[\sum_{k=0}^{N}c_k\Big(\max\{x-\gp_k,0\}-\max\{\gp_{n}-\gp_k,0\}\Big)\right]\\
&=\sum_{k=0}^{n}c_k [(x-\gp_k)-(\gp_{n}-\gp_k)]= \sum_{k=0}^{n}c_k (x-\gp_{n})=a_n(x-\gp_n).\end{split}
\end{align}
This
shows
for all
$n\in [0,N-1]\cap\Z$ that
  $g $ is affine linear on the interval $ (\gp_n,\gp_{n+1}]$.
This, 
the fact that for all
$n\in [0,N-1]\cap\Z$ it holds that
  $f $ is affine linear on the interval $ (\gp_n,\gp_{n+1}]$, 
the fact that $\forall\, x\in (-\infty,\gp_0]\colon f(x)= g(x)=a_0$, and an induction argument imply 
for all $x\in (-\infty,\gp_N]$ 
that $f(x)=g(x)$. Furthermore, 
\eqref{x04},
the fact that
$\forall \,n\in [0,N-1]\cap\Z\colon   \gp_n<\gp_{n+1}$, and the fact that
$\sum_{k=0}^{N}c_k=0$
imply for all $x\in (\gp_N,\infty)$ that
\begin{align}\begin{split}
g(x)-g(\gp_N)&= \left[\sum_{k=0}^{N}c_k\Big(\max\{x-\gp_k,0\}-\max\{\gp_{N}-\gp_k,0\}\Big)\right]\\
&=\sum_{k=0}^{N}c_k [(x-\gp_k)-(\gp_{N}-\gp_k)]= \sum_{k=0}^{N}c_k (x-\gp_{N})=0.\end{split}
\end{align}
This shows  for all $x\in (\gp_N,\infty)$ that $g(x)=g(\gp_N)$.
This, the fact that $\forall \,x\in (\gp_N,\infty)\colon f(x)=f(\gp_N)$,
and the fact that
$\forall\,x\in (-\infty,\gp_N]\colon f(x)=g(x)$, and \eqref{x04}
prove for all $x\in\R$ that 
\begin{align}
f(x)=g(x)=f(\gp_0)+\sum_{k=0}^{N}c_k\max\{x-\gp_k,0\}.\label{x04b}
\end{align}
Next, the definition of $\funcCalR$ and $\funcCalL$ (see
\eqref{m07c} and \eqref{m07d}), \eqref{x01},  \eqref{x02}, and \eqref{x04b} imply that
for all $x\in\R$ it holds that $\funcCalL(\Phi)=(1,N+1,1)$ and
\begin{align}\begin{split}
&(\funcCalR(\Phi))(x)=
W_2(
\funcBoldA{N+1}(W_1x+B_1))+B_2\\&= 
\begin{pmatrix}
c_0&c_1&\ldots&c_N
\end{pmatrix}
\begin{pmatrix}
\max\{x-\gp_0,0\}\\
\max\{x-\gp_1,0\}\\
\vdots\\
\max\{x-\gp_N,0\}
\end{pmatrix}+f(\gp_0)=
f(\gp_0)+\sum_{k=0}^{N}c_k\max\{x-\gp_k,0\}=f(x).\end{split}
\end{align}
This establishes \eqref{x06}.
The proof of \cref{xn11} is thus completed.
\end{proof}
\begin{lemma}\label{xn01}
Let $L\in [0,\infty)$, $N\in\N$,
$a\in\R$, $b\in (a,\infty)$,
$\gp_0, \gp_1,\ldots, \gp_N\in \R$ satisfy for all $n\in [0,N]\cap\Z$ that $\gp_n=a+\frac{(b-a)n}{N}$,
let
$f\colon \R\to\R$ satisfy for all $x,y\in\R$ that
\begin{align}\label{xn02}
|f(x)-f(y)|\leq L|x-y|,
\end{align}
and let $g\colon \R \to \R$ satisfy
for all $x\in \R$, $n\in [0,N-1]\cap\Z$ that 
\begin{equation}\label{xn03}
g(x)=
\begin{cases}
f(\gp_0) &
\colon x\in (-\infty,\gp_0]\\[1ex]
\frac{f(\gp_n)(\gp_{n+1}-x)
+f(\gp_{n+1})(x-\gp_{n})}{\gp_{n+1}-\gp_n} & \colon x\in (\gp_n,\gp_{n+1}]\\[1ex]
f(\gp_N) & \colon  x\in (\gp_N,\infty).
\end{cases}
\end{equation}
Then
\begin{enumerate}[(i)]
\item \label{xn08b}it holds for all $n\in [0,N]\cap\Z$ that $g(\gp_n)=f(\gp_n)$,
\item\label{xn07} it holds for all $x,y\in\R$ that $|g(x)-g(y)|\leq L|x-y|$, and
\item\label{xn08} it holds for all $x\in[a,b]$ that $|g(x)-f(x)|\leq \frac{2L(b-a)}{N}$.
%\item \label{xn10}it holds that $g\in \funcCalR(\{\Phi\in\setCalN\colon \funcCalL(\Phi)=(1,N+1,1)\})$.
\end{enumerate}
\end{lemma}

\begin{proof}[Proof of \cref{xn01}]
Throughout this proof 
let $r,\ell\colon \R \to\R$ satisfy
for all $x\in \R \setminus (a,b]$ that
\begin{align}\label{xn04}
 r(x)=\ell(x)=x
\end{align}
and which satisfies for all $n\in [0,N-1]\cap\Z$, $x\in (\gp_n,\gp_{n+1}]$ that
\begin{align}
r(x)= \gp_{n+1}\quad\text{and}\quad
\ell(x)=  \gp_{n}.
\end{align}
%In the rest of the proof we verify that $g$ satisfies \cref{xn08b,xn07,xn08,xn10} ({\color{red}allows to write this?}).
\cref{xn08b} follows from \cref{xn03}.
Next, observe that for all $x,y\in (a,b]$ with $x\leq y$ and $\ell(y)<r(x)$ it holds that $r(x)=r(y)$ and $\ell(x) =\ell(y)$. This, 
\eqref{xn04},  \eqref{xn03}, and \eqref{xn02}
show that
for all $x,y\in\R$ with 
$x\leq y$ and $\ell(y)<r(x)$ it holds that 
$x,y\in (a,b]$,
$r(x)=r(y)$, $\ell(x) =\ell(y)$, and
\begin{align}
|g(x)-g(y)|= \left|\frac{f(r(x))-f(\ell(x))}{r(x)-\ell(x)} (x-y)\right|\leq L |x-y|.\label{xn06}
\end{align}
Furthermore,  \eqref{xn03}, \eqref{xn02}, and
the fact that $\forall\,x\in\R \colon \ell(x)\leq x\leq r(x)$ imply 
for all $x\in (a,b]$ that
\begin{align}\begin{split}
|g(x)-g(r(x))|
&=\left| \frac{f(\ell(x))-f(r(x))}{\ell(x)-r(x)} (x-\ell(x))+f(\ell(x))-f(r(x))\right|\\
&= \left| \frac{(f(\ell(x))-f(r(x)))(x-r(x))}{\ell(x)-r(x)}\right|\leq L|x-r(x)|=L(r(x)-x),
\end{split}\end{align}
and
\begin{align}\begin{split}
g(x)-g(\ell(x))&=\left| \frac{f(\ell(x))-f(r(x))}{\ell(x)-r(x)} (x-\ell(x))+f(\ell(x))-f(\ell(x))\right|\\
&=\left| \frac{f(\ell(x))-f(r(x))}{\ell(x)-r(x)} (x-\ell(x))\right|\leq L|x-\ell(x)|=L(x-\ell(x)).
\end{split}\end{align}
This and \eqref{xn04}
show for all 
$x\in \R$ that
\begin{align}\label{xn09}
|g(x)-g(r(x))|\leq L(x-r(x)) \quad\text{and}\quad
|g(x)-g(\ell(x))|\leq L(x-\ell(x)) .
\end{align}
The triangle inequality therefore shows for all $x,y\in \R$ with 
$x\leq y$ and
$r(x)\leq \ell (y)$ that
\begin{align}\begin{split}
|g(x)-g(y)|&\leq | g(x)-g(r(x))|+|g(r(x))-g(\ell(y))|+|g(\ell(y))-g(y)|\\
&\leq L (r(x)-x)+ L(\ell(y)-r(x))+ L (y-\ell(y))= L(y-x)= L|y-x|.\end{split}
\end{align}
This and \eqref{xn06}
show for all $x,y\in\R$ with $x\leq y$ that $|g(x)-g(y)|\leq L|x-y| $. Symmetry hence establishes \cref{xn07}.
Next, the fact that $\forall\,x\in\R\colon  g(\ell(x))=f(x)$, the triangle inequality,
\eqref{xn02}, \eqref{xn09}, and the fact that
$\forall\,x\in [a,b]\colon 0\leq x-\ell (x)\leq 1/N$
 imply for all $x\in[a,b]$ that
\begin{align}\begin{split}
|g(x)-f(x)|&= |g(x)-f(\ell(x))+f(\ell(x))-f(x)|\\
&= |g(x)-g(\ell(x))+f(\ell(x))-f(x)|\\
&\leq 
 |g(x)-g(\ell(x))|+|f(\ell(x))-f(x)|\leq 2L (x-\ell(x))\leq 2L(b-a)/N.
\end{split}\end{align}
This establishes \cref{xn08}.
The proof of \cref{xn01} is thus completed.
\end{proof}

\begin{corollary} \label{z01}
Assume \cref{m07}, let $\epsilon \in (0,1)$, $L\in [0,\infty)$, $q\in (1,\infty)$,
and let
$f\colon \R\to\R$ satisfy for all $x,y\in\R$ that
$|f(x)-f(y)|\leq L|x-y|. $
Then there exists a function $g\colon \R\to\R$ such that
\begin{enumerate}[(i)]
\item \label{z01a}it holds for all $x,y\in\R$ that $|g(x)-g(y)|\leq L|x-y|$,
\item \label{z01b} it holds for all $x\in\R$ that
$|f(x)-g(x)|\leq \epsilon (1+|x|^q)$, and
\item \label{z01c} it holds that
\begin{equation}
g\in \funcCalR\left(\left\{\Phi\in\setCalN\colon \funcCalL(\Phi)\in \N^3 \text{ and } \supnorm{\funcCalL(\Phi)}\leq \frac{16\left[\max\!\left\{1,\left|L\big(4L+2|f(0)|\big)\right|^\frac{1}{(q-1)}\right\}\right]}{\epsilon ^{\frac{q}{(q-1)}}}\right\}\right).
\end{equation}
\end{enumerate}
\end{corollary}
\begin{proof}[Proof of \cref{z01}]
Throughout this proof let $R\in \R$, $N\in\N$ satisfy that
\begin{align}
 \frac{4L+2|f(0)|}{R^{q-1}}=\epsilon\quad\text{and}\quad
N=\inf\left\{n\in\N\cap[2,\infty)\colon 
\frac{4LR}{n}\leq \epsilon\right\},\label{x10}
\end{align}
let $\gp_0, \gp_1,\ldots, \gp_N\in \R$ be the real numbers which satisfy for all $n\in [0,N]\cap\Z$ that $\gp_n=R(-1+\frac{2n}{N})$,
and let $g\colon \R \to \R$ satisfy
for all $x\in \R$, $n\in [0,N-1]\cap\Z$ that 
\begin{equation}\label{xn03b}
g(x)=
\begin{cases}
f(\gp_0) &
\colon x\in (-\infty,\gp_0]\\[1ex]
\frac{f(\gp_n)(\gp_{n+1}-x)
+f(\gp_{n+1})(x-\gp_{n})}{\gp_{n+1}-\gp_n} & \colon x\in (\gp_n,\gp_{n+1}]\\[1ex]
f(\gp_N) & \colon  x\in (\gp_N,\infty).
\end{cases}
\end{equation}
%and
%let  $g\colon\R \to\R$ 
%be a function, whose existence is ensured by
%\cref{xn01} (applied with 
%$ a=-R$, $b=R$, $N=N$, $f=f$, and $L=L$ in the notation of \cref{xn01})
%such that it holds that $f(R)=g(R)$ and $f(-R)=g(-R)$
%for all $x,y\in\R$ it holds  that $|g(x)-g(y)|\leq L|x-y|$,
%for all $x\in[-R,R]$ it holds  that $|g(x)-f(x)|\leq 4LR/N$, and
%it holds that $g\in \funcCalR(\{\Phi\in\setCalN\colon \funcCalL(\Phi)=(1,N+1,1)\})$. 
By \cref{xn07} in \cref{xn01} the function $g$ satisfies \cref{z01a}.
 Next, it follows from \cref{xn08} in \cref{xn01}
that for all $x\in[-R,R]$ it holds
$ |g(x)-f(x)|\leq 4LR/N$. This and
the fact that 
$4LR/N\leq \eps$ prove that for all $x\in[-R,R]$ it holds  that $|g(x)-f(x)|\leq \epsilon\leq \epsilon (1+|x|^q)$.
Next, the triangle inequality,
the fact that $f(R)=g(R)$, and the Lipschitz condition of $f$ and $g$ imply for all $x\in\R$ that
\begin{align}\begin{split}
|f(x)-g(x)|&\leq 
|f(x)-f(0)|+|f(0)|+|g(x)-g(R)|+|g(R)|\\
&=
|f(x)-f(0)|+|f(0)|+|g(x)-g(R)|+|f(R)|\\
&\leq
|f(x)-f(0)|+|f(0)|+|g(x)-g(R)|+|f(R)-f(0)|+|f(0)|\\
&\leq L|x|+2|f(0)|+L|x-R|+LR\\
&\leq 2L(|x|+R)+2|f(0)|.\end{split}
\end{align}
This shows for all $x\in\R\setminus [-R,R]$ that
\begin{align}\begin{split}
\frac{|f(x)-g(x)|}{1+|x|^q}
&\leq \frac{2L(|x|+R)+2|f(0)|}{1+|x|^q}\leq \frac{4L|x|+2|f(0)|}{1+|x|^q}
\\
&\leq \frac{4L}{|x|^{q-1}}+\frac{2|f(0)|}{|x|^q}\leq 
\frac{4L}{R^{q-1}}+\frac{2|f(0)|}{R^q}\leq \frac{4L+2|f(0)|}{R^{q-1}}
= \epsilon.\end{split}
\end{align}
This and the fact that 
$\forall\,x\in[-R,R]\colon |g(x)-f(x)|\leq \epsilon (1+|x|^q)$ prove that for all 
$x\in\R$ it holds that $|g(x)-f(x)|\leq  \epsilon (1+|x|^q)$.
%let $g\colon \R\to\R$ 
This shows that $g$ satisfies \cref{z01b}.
Next, \cref{xn08b} in \cref{xn01} ensures that $g$ satisfies for all $x\in (-\infty,-R]$ 
that $g(x)=g(-R)$, for all $n\in [0,\N-1]\cap \Z$, $x\in (\gp_n,\gp_{n+1}]$ that 
$g(x)=g(\gp_n)+\frac{g(\gp_{n+1})-g(\gp_n)}{\gp_{n+1}-\gp_n}(x-\gp_n)$, and
for all $x\in (R,\infty)$ that $g(x)=g(R)$.
This and \cref{xn11} (with $N=N$, 
$f=g$,
$\gp_n=\gp_n $ for $n\in[0,N]\cap\Z$, and $a_n= (g(\gp_{n+1})-g(\gp_n))/(\gp_{n+1}-\gp_n)$ for $n\in[0,N-1]\cap\Z$ in the notation of \cref{xn11}) imply that 
\begin{equation}\label{x11}
g\in \funcCalR(\{\Phi\in\setCalN\colon \funcCalL(\Phi)=(1,N+1,1)\}).
\end{equation}
Furthermore, \eqref{x10} implies that
$\frac{4LR}{N-1}>\epsilon $. This and the fact that $N\geq 2$ imply that
$\frac{N}{2}\leq N-1<\frac{4LR}{\epsilon}$. This and \cref{x10} ensure that 
\begin{align}
N \leq \frac{8LR}{\epsilon}= \frac{8L}{\epsilon} \left(\frac{4L+2|f(0)|}{\epsilon }\right)^\frac{1}{q-1}
=8L\big(4L+2|f(0)|\big)^\frac{1}{q-1}\epsilon ^{-\frac{q}{q-1}}.
\end{align}
This and \cref{x11} establish \cref{z01c}.
The proof of \cref{z01} is thus completed.
\end{proof}

\section{Deep neural network approximations for PDEs}\label{sec:main_result}
\subsection{Deep neural network approximations with specific polynomial convergence rates}
\sloppy
\begin{theorem}\label{n18}
Let $\left \|\cdot \right \|, \supnorm{\cdot} \colon (\cup_{d\in \N} \R^d) \to [0,\infty)$ and $\cardna \colon (\cup_{d\in \N}\R^d) \to \N$ satisfy for all $d\in \N$, $x=(x_1,\ldots,x_d)\in \R^d$ that $\|x\|=[\sum_{i=1}^d(x_i)^2]^{1/2}$, $\supnorm{x}=\max_{i\in [1,d]\cap \N}|x_i|$,
and $\card{x}=d$,
let
$T,\LipConstF, \boundFG,\beta \in [0,\infty)$, $p,\mathfrak{p}\in \N$, $q\in \N\cap [2,\infty)$, $\alpha \in [2,\infty)$,
  let
$f\colon \R\to\R$ satisfy for all $x,y\in\R$ that
$|f(x)-f(y)|\leq \LipConstF |x-y|$,
for every $d\in \N$ let $g_d\in C(\R^d,\R)$,
for every $d\in \N$
let $\nu_d\colon \mathcal B(\R^d)\to [0,1]$ be a probability measure on $(\R^d, \mathcal B(\R^d))$,
for every $d\in \N$ let $\funcBoldA{d}\colon \R^d\to\R^d $ satisfy for all %$d\in\N$,
 $x=(x_1,\ldots,x_d)\in \R^d$ that 
\begin{align}
\funcBoldA{d}(x)= \left(\max\{x_1,0\},\ldots,\max\{x_d,0\}\right),
\end{align}
let $\setCalD=\cup_{H\in \N} \N^{H+2}$,
let
\begin{align}\begin{split}
\setCalN= \bigcup_{H\in  \N}\bigcup_{(k_0,k_1,\ldots,k_{H+1})\in \N^{H+2}}
\left[ \prod_{n=1}^{H+1} \left(\R^{k_{n}\times k_{n-1}} \times\R^{k_{n}}\right)\right],
%\quad\text{and}\quad
%\setCalD=\bigcup_{H\in \N} \N^{H+2},
\end{split}
\end{align}
let
$\funcCalP\colon\setCalN\to \N$,
$\funcCalL\colon \setCalN\to\setCalD$, and
$\funcCalR\colon \setCalN\to (\cup_{k,l\in \N} C(\R^k,\R^l))$
%\begin{align}
%\funcCalL\colon \setCalN\to\setCalD\quad\text{and}\quad
%\funcCalR\colon \setCalN\to \bigcup_{k,l\in \N} C(\R^k,\R^l)\label{m07b}
%\end{align}
satisfy
for all $H\in \N$, $k_0,k_1,\ldots,k_H,k_{H+1}\in \N$,
$
\Phi = ((W_1,B_1),\ldots,(W_{H+1},B_{H+1}))\in \prod_{n=1}^{H+1} \left(\R^{k_n\times k_{n-1}} \times\R^{k_n}\right), 
$
$x_0 \in \R^{k_0},\ldots,x_{H}\in \R^{k_{H}}$ with 
$\forall\, n\in \N\cap [1,H]\colon x_n = \funcBoldA{k_n}(W_n x_{n-1}+B_n )$ 
that
\begin{equation}\label{n02}
\funcCalP(\Phi )=\sum_{n=1}^{H+1}k_n(k_{n-1}+1), \qquad
\funcCalL(\Phi)= (k_0,k_1,\ldots,k_{H}, k_{H+1}),
\end{equation}
\begin{equation}
\funcCalR(\Phi )\in C(\R^{k_0},\R ^ {k_{H+1}}),\qquad\text{and}\qquad 
 (\funcCalR(\Phi)) (x_0) = W_{H+1}x_{H}+B_{H+1},
\end{equation}
for every $\varepsilon\in (0,1]$, $d\in \N$ let  $\netapproxg{d}{\varepsilon} \in \setCalN$,
assume for all $d\in \N$,
$x\in \R^d$, $\varepsilon \in (0,1]$ that
$\funcCalR(\netapproxg{d}{\varepsilon})\in C(\R^d,\R)$, 
$|(\funcCalR(\netapproxg{d}{\varepsilon}))(x)|\le \boundFG d^p (1+\normRd{x})^p$,
$\left| g_d(x)-(\funcCalR(\netapproxg{d}{\varepsilon}))(x)\right| \le \varepsilon \boundFG d^p (1+\normRd{x})^{pq}$,
$\supnorm{\funcCalL(\netapproxg{d}{\varepsilon})}\le \boundFG d^p\varepsilon^{-\alpha}$,
$\cardL{\netapproxg{d}{\varepsilon}}\le \boundFG d^p\varepsilon^{-\beta}$, and
$\left(\int_{\R^d}\normRd{y}^{2pq} \nu_d(dy)\right)^{1/(2pq)}\le \boundFG d^{\mathfrak p}$.
Then
\begin{enumerate}[(i)]
\item \label{n07} there exist unique $u_d\in C([0,T]\times \R^d,\R)$, $d\in \N$, such that 
for every $d\in \N$, $x\in \R^d$, $s\in [0,T]$,
every probability space
$(\Omega, \mathcal{F}, \P)$, and every standard 
%$d$-dimensional
Brownian motion $\fwpr \colon [0,T]\times \Omega\to \R^d$ with continuous sample paths it holds that
$\sup_{t\in [0,T]}\sup_{y\in \R^d}\big(\frac{|u_d(t,y)|}{1+\normRd{y}^p} \big)<\infty$ and
\begin{align}\label{n03}
\smallU_d(s,x)=\E\!\left[\funcG_d\left(x+\fwpr_{T-s}\right)+\int_s^T f\left(\smallU_d\left(t,x+\fwpr_{t-s}\right)\right)\,dt\right]
\end{align}
and
\item \label{n08} there exist $(\Psi_{d,\varepsilon})_{d\in \N,\varepsilon \in (0,1]}\subseteq \setCalN$, $\eta \in (0,\infty)$, $C=(C_\gamma)_{\gamma\in (0,1]}\colon (0,1]\to (0,\infty)$ such that for all
$d\in \N$, $\varepsilon, \gamma \in (0,1]$ it holds that
$\funcCalR(\Psi_{d,\varepsilon})\in C(\R^d,\R)$, 
%\begin{equation}\label{n14}
$\funcCalP(\Psi_{d,\varepsilon})\le C_\gamma d^\eta \varepsilon^{-(4+2\alpha+\beta+\gamma)}$,
%\end{equation}
and 
\begin{equation}\label{n15}
\left[\int_{\R^d}\left|\smallU_d(0,x)-(\funcCalR(\Psi_{d,\varepsilon}))(x)\right|^2\nu_d(dx)\right]^{\!\!\nicefrac{1}{2}}\le \varepsilon.
\end{equation}
\end{enumerate}
\end{theorem}

\begin{proof}[Proof of \cref{n18}]
Throughout the proof assume without loss of generality that
\begin{equation}\label{n21e}
\boundFG \ge \max\!\left\{16,|f(0)|+1,16\left|L\big(4L+2|f(0)|\big)\right|^\frac{1}{(q-1)}\right\}.
\end{equation}
Note that
the triangle inequality, the fact that $\forall \, d\in \N, x \in \R^d, \varepsilon \in (0,1]\colon 
|(\funcCalR(\netapproxg{d}{\varepsilon}))(x)|\le \boundFG d^p(1+\normRd{x})^p
$, and the
fact that $\forall \, d\in \N, x \in \R^d, \varepsilon \in (0,1]\colon 
\left| g_d(x)-(\funcCalR(\netapproxg{d}{\varepsilon}))(x)\right| \le \varepsilon \boundFG d^p (1+\normRd{x})^{pq}$
imply for all $d\in \N$, $x \in \R^d$, $\varepsilon \in (0,1]$ that
\begin{equation}
\begin{split}
\left|g_d(x)\right|&\leq \left|g_d(x)-(\funcCalR(\netapproxg{d}{\varepsilon}))(x)\right|+\left|(\funcCalR(\netapproxg{d}{\varepsilon}))(x)\right|
\le \varepsilon \boundFG d^p (1+\normRd{x})^{pq}+ \boundFG d^p (1+\normRd{x})^p.
\end{split}
\end{equation}
This proves for all $d\in \N$, $x \in \R^d$ that 
\begin{equation}\label{n21}
\left|g_d(x)\right|\leq \boundFG d^p (1+\normRd{x})^p.
\end{equation}
Corollary 3.11 in \cite{HJK+18}, the fact that $f$ is globally Lipschitz continuous, and \cref{n21}
hence establish \cref{n07}.
It thus remains to prove \cref{n08}.
To this end note that \cref{z01} ensures that 
there exist
$\netapproxf{\varepsilon}\in \setCalN$, $\varepsilon\in (0,1]$, which satisfy
for all
$v,w\in \R$, $\varepsilon \in (0,1]$ that
$\funcCalR(\netapproxf{\varepsilon})\in C(\R,\R)$,
$
\left|(\funcCalR(\netapproxf{\varepsilon}))(w)-(\funcCalR(\netapproxf{\varepsilon}))(v)\right|\leq \LipConstF\left|w-v\right|
$, 
$\left |f(v)-(\funcCalR(\netapproxf{\varepsilon}))(v)\right|\le \varepsilon (1+|v|^q)$,
$\cardL{\netapproxf{\varepsilon}}=3$, and
\begin{equation}\label{n21c}
 \supnorm{\funcCalL(\netapproxf{\varepsilon})}\leq 16\left[\max\!\left\{1,\left|L\big(4L+2|f(0)|\big)\right|^\frac{1}{(q-1)}\right\}\right]\varepsilon ^{-\frac{q}{(q-1)}}.
\end{equation}
Note that the fact that $\boundFG\ge 1+|f(0)|$ implies for all $\varepsilon\in (0,1]$ that
\begin{equation}\label{n21b}
\left|(\funcCalR(\netapproxf{\varepsilon}))(0)\right|\le \left|(\funcCalR(\netapproxf{\varepsilon}))(0)-f(0)\right|+|f(0)|\le \varepsilon+|f(0)|\le \boundFG.
\end{equation}
Next let
$(\Omega, \mathcal{F}, \P)$
be a probability space, 
for every $d\in \N$
let $\fwpr^d \colon [0,T]\times \Omega\to \R^d$ be a standard
% $d$-dimensional 
Brownian motion with continuous sample paths,
let
$  \Theta = \bigcup_{ n \in \N } \Z^n$,
let $\unif^\theta\colon \Omega\to[0,1]$, $\theta\in \Theta$, be independent random variables which are uniformly distributed on $[0,1]$, 
let $\uniform^\theta\colon [0,T]\times \Omega\to [0, T]$, $\theta\in\Theta$, satisfy 
for all $t\in [0,T]$, $\theta\in \Theta$ that 
$\uniform^\theta _t = t+ (T-t)\unif^\theta$,
for every $d\in \N$
let $\sppr^{\theta,d}\colon[0,T]\times \Omega \to\R^d $, $\theta\in \Theta$, be independent 
%$d$-dimensional 
standard Brownian motions with continuous sample paths,
assume for every $d\in \N$ that $(\unif^\theta)_{\theta\in \Theta}$ and $(\sppr^{\theta,d})_{\theta\in \Theta}$ are independent,
and
let
$ 
  {\bigU}_{ n,M,d,\delta}^{\theta} \colon [0, T] \times \R^d \times \Omega \to \R
$, $n,M\in\Z$, $d\in \N$, $\delta \in (0,1]$, $\theta\in\Theta$, satisfy
for all $d,n,M \in \N$, $\delta \in (0,1]$, $\theta\in\Theta $, 
$t \in [0,T]$, $x\in\R^d $
that ${\bigU}_{-1,M,d,\delta}^{\theta}(t,x)={\bigU}_{0,M,d,\delta}^{\theta}(t,x)=0$ and 
%\begin{equation}  \begin{split}
\begin{equation}
\begin{split}
  {\bigU}_{n,M,d,\delta}^{\theta}(t,x)
=&
  \frac{1}{M^n}
 \sum_{i=1}^{M^n} 
      (\funcCalR(\netapproxg{d}{\delta}))\big(x+\sppr^{(\theta,0,-i),d}_{T}-\sppr^{(\theta,0,-i),d}_{t}\big)
 \\
 &+
  \sum_{l=0}^{n-1} \frac{(T-t)}{M^{n-l}}
    \Bigg[\sum_{i=1}^{M^{n-l}}
      \big((\funcCalR(\netapproxf{\delta}))\circ{\bigU}_{l,M,d,\delta}^{(\theta,l,i)}-\1_{\N}(l)(\funcCalR(\netapproxf{\delta}))\circ {\bigU}_{l-1,M,d,\delta}^{(\theta,-l,i)}\big)\\
&\qquad\qquad\qquad\qquad\qquad\qquad\qquad\qquad
      \left(\uniform_t^{(\theta,l,i)},x+\sppr_{\uniform_t^{(\theta,l,i)}}^{(\theta,l,i),d}-\sppr_{t}^{(\theta,l,i),d}\right)
    \Bigg],
\end{split}\label{n04}
\end{equation}
let $c_{d}\in [1,\infty)$, $d\in \N$, satisfy for all $d\in \N$ that
\begin{equation}
c_{d}= \left(e^{LT} (T+1)\right)^{q+1}\left((\boundFG d^p)^q+1\right)
\left[
1+\left(\int_{\R^d}\normRd{x}^{2pq} \nu_d(dx)\right)^{\nicefrac{1}{(2pq)}}
+\left(\Exp{           \normRd{\fwpr^d_T}^{pq}\Big.\! }\bigg.\!\right)^{\nicefrac{1}{(pq)}}
\right]^{pq},
\end{equation}
let $k_{d,\varepsilon}\in \N$, $d\in \N$, $\varepsilon \in (0,1]$, satisfy for all $d\in \N$, $\varepsilon \in (0,1]$ that
\begin{align}\label{n10}
k_{d,\varepsilon}=\max\left\{ \supnorm{\funcCalL(\netapproxf{\varepsilon})},\supnorm{\funcCalL(\netapproxg{d}{\varepsilon})},2\right\},
\end{align}
let $\tilde C=(\tilde C_\gamma)_{\gamma\in (0,1]}\colon (0,\infty)\to (0,\infty]$
satisfy for all $\gamma\in (0,\infty)$ that 
\begin{equation}
\tilde C_\gamma=
\sup_{n\in \N \cap [2,\infty)}
\left[
n (3 n)^{2n}
\left(\frac{\sqrt{e}(1+2LT)}{\sqrt{n-1}} \right)^{(n-1)(4+\gamma)}
\right],
\end{equation}
let $N_{d,\varepsilon}\in \N$, $d\in \N$, $\varepsilon \in (0,1]$, satisfy for all $d\in \N$, $\varepsilon \in (0,1]$ that
\begin{equation}\label{n12}
N_{d,\varepsilon}=\min\!\left\{n \in \N \cap [2,\infty) \colon \left[ c_d
\left(\frac{\sqrt{e}(1+2LT)}{\sqrt{n}} \right)^n\right] \le \frac \varepsilon 2\right\},
\end{equation}
and let $\delta_{d,\varepsilon}\in (0,1]$, $d\in \N$, $\varepsilon \in (0,1]$, satisfy for all $d\in \N$, $\varepsilon \in (0,1]$ that $\delta_{d,\varepsilon}=\frac{\epsilon}{4\boundFG d ^p c_d}$.
Note that the fact that
for all $d\in \N$ the random variable $\normRd{\nicefrac{\fwpr^d_T}{\sqrt{T}}}^{2}$ is chi-squared distributed with $d$ degrees of freedom and Jensen's inequality imply that for all $d\in \N$ it holds that
\begin{equation}
\left(\E\!\left[\normRd{\fwpr^d_T}^{pq}\right]\right)^2\le \E\!\left[\normRd{\fwpr^d_T}^{2pq}\right]=(2T)^{pq} \left[\frac{\Gamma\!\left(\frac{d}{2}+pq\right)}{\Gamma\!\left(\frac{d}{2}\right)}\right]
=(2T)^{pq} \left[\prod_{k=0}^{pq-1}\left(\frac{d}{2}+k \right)\right].
\end{equation}
This implies for all $d\in \N$ that
\begin{equation}
\left(\Exp{\normRd{\fwpr^d_T}^{pq}\Big.\! }\bigg.\!\right)^{\nicefrac{1}{(pq)}}
=
\left(\Exp{ \normRd{\fwpr^d_T}^{pq}\Big.\! }\bigg.\!\right)^{\nicefrac{2}{(2pq)}}
\le 
\sqrt{2T}\left( \prod_{k=0}^{pq-1}\left(\frac{d}{2}+k \right)\right)^{\nicefrac{1}{(2pq)}}
\le 
\sqrt{2T\left(\frac{d}{2}+pq-1\right)}.
\end{equation}
This together with the fact that $\forall \, d\in \N\colon \left(\int_{\R^d}\normRd{x}^{2pq} \nu_d(dx)\right)^{\nicefrac{1}{(2pq)}}\le \boundFG d^{\mathfrak p}$ implies that
there exist $\bar C \in (0,\infty)$ such that for all $d\in \N$ it holds that
\begin{equation}\label{n17}
c_d\le \bar C d^{pq}\left(\frac{1+d^{\mathfrak p}+\sqrt{d}}{3} \right)^{pq}\le \bar C d^{(\mathfrak p+1) pq}.
\end{equation}
Next note that for all $\gamma\in (0,\infty)$ it holds that
\begin{equation}\label{n16}
\begin{split}
\tilde C_\gamma&=
\sup_{n\in \N \cap [2,\infty)}
\left[
n (3 n)^{2n}
\left(\frac{\sqrt{e}(1+2LT)}{\sqrt{n-1}} \right)^{(n-1)(4+\gamma)}
\right]\\
&=
\sup_{n\in \N \cap [2,\infty)}
\left[
(\sqrt{e}(1+2LT))^{(n-1)(4+\gamma)}n^3 3^{2n} (n-1)^{-(n-1)\frac{\gamma}{2}}
\left(
\frac{n}{n-1}
\right)^{2(n-1)}
\right]\\
&\leq 
\left[
\sup_{n\in \N \cap [2,\infty)}
\left[
(\sqrt{e}(1+2LT))^{(n-1)(4+\gamma)}n^3 3^{2n} (n-1)^{-(n-1)\frac{\gamma}{2}}
\right]
\right]
\left[
\sup_{n\in \N \cap [2,\infty)}
\left(
\frac{n}{n-1}
\right)^{2(n-1)}
\right]\\
&
<\infty.
\end{split}
\end{equation}
The fact that 
for all $d\in \N$,
$v\in \R$, $x\in \R^d$, $\varepsilon \in (0,1]$ it holds that
$\left |f(v)-(\funcCalR(\netapproxf{\varepsilon}))(v)\right|\le \varepsilon (1+|v|^q)$
and $\left| g_d(x)-(\funcCalR(\netapproxg{d}{\varepsilon}))(x)\right| \le \varepsilon \boundFG d^p (1+\normRd{x})^{pq}$ implies that for all $d\in \N$,
$v\in \R$, $x\in \R^d$, $\varepsilon \in (0,1]$ it holds that
\begin{equation}
\begin{split}
\max\left\{\left |f(v)-(\funcCalR(\netapproxf{\varepsilon}))(v)\right|, \left| g_d(x)-(\funcCalR(\netapproxg{d}{\varepsilon}))(x)\right|\right\} 
&\le \max \left\{ \varepsilon (1+|v|^q),  \varepsilon \boundFG d^p (1+\normRd{x})^{pq} \right\}\\
&\le  \varepsilon \boundFG d^p ((1+ \normRd{x})^{pq}+|v|^q).
\end{split}
\end{equation}
This, \cref{n21}, \cref{n21b}, the fact that for all $d\in \N$,
$v,w\in \R$, $x\in \R^d$, $\varepsilon \in (0,1]$ it holds that
$|f(v)-f(w)|\le \LipConstF$,
$
\left|(\funcCalR(\netapproxf{\varepsilon}))(v)-(\funcCalR(\netapproxf{\varepsilon}))(w)\right|\leq \LipConstF\left|v-w\right|
$, 
$
|f(0)|\le \boundFG
$,
$|(\funcCalR(\netapproxg{d}{\varepsilon}))(x)|\le \boundFG d^p (1+\normRd{x})^p$,
and \cref{m06} (with $f_1=f$, $f_2=\funcCalR(\netapproxf{\delta})$, $g_1=g_d$, $g_2=\funcCalR(\netapproxg{d}{\delta})$, $L=L$, $\delta=\delta \boundFG d^p$, $\boundFG=\boundFG d^p$, $\fwpr=\fwpr^d$ in the notation of \cref{m06})
imply that for all $d,N,M\in \N$, $\delta \in (0,1]$ it holds that
\begin{multline}
\left(\int_{\R^d}\E\!\left[\left|U^0_{N,M,d,\delta}(0,x)-\smallU_d(0,x)\right|^2\right]\nu_d(dx)\right)^{\!\!\nicefrac{1}{2}}
\\\leq 
\left(e^{LT} (T+1)\right)^{q+1}\left((\boundFG d^p)^q+1\right)\left(\delta \boundFG d^p+\frac{e^{M/2}(1+2LT)^{N}}{M^{N/2}}\right)\\
\cdot 
\left[\int_{\R^d}
  \left(1+\normRd{x}+
\left(\Exp{           \normRd{\fwpr^d_T}^{pq}\Big.\! }\bigg.\!\right)^{\!\!\frac{1}{pq}}\right)^{2pq}
\nu_d(dx)\right]^{\!\!\nicefrac{1}{2}}.
\end{multline}
The triangle inequality hence proves for all $d,N,M\in \N$, $\delta \in (0,1]$ that
\begin{align}\begin{split}
&\left(\int_{\R^d}\E\!\left[\left|U^0_{N,M,d,\delta}(0,x)-\smallU_d(0,x)\right|^2\right]\nu_d(dx)\right)^{\!\!\nicefrac{1}{2}}
\\&\leq 
\left(e^{LT} (T+1)\right)^{q+1}\left((\boundFG d^p)^q+1\right)\left(\delta \boundFG d^p+\frac{e^{M/2}(1+2LT)^{N}}{M^{N/2}}\right)\\
&\qquad \cdot 
\left[
1+\left(\int_{\R^d}\normRd{x}^{2pq} \nu_d(dx)\right)^{\nicefrac{1}{(2pq)}}
+\left(\Exp{           \normRd{\fwpr^d_T}^{pq}\Big.\! }\bigg.\!\right)^{\nicefrac{1}{(pq)}}
\right]^{pq}
\\
&=
c_d\left(\delta \boundFG d^p +  \frac{e^{M/2}(1+2LT)^{N}}{M^{N/2}}\right)
.
\end{split}\label{n05}
\end{align}
This and Fubini's theorem imply that for all $d\in \N$, $\varepsilon \in (0,1]$ it holds that
\begin{align}\begin{split}
&\E\!\left[\int_{\R^d}\left|U^0_{N_{d,\varepsilon},N_{d,\varepsilon},d,\delta_{d,\varepsilon}}(0,x)-\smallU_d(0,x)\right|^2\nu_d(dx)\right]\\
&=
\int_{\R^d}\E\!\left[\left|U^0_{N_{d,\varepsilon},N_{d,\varepsilon},d,\delta_{d,\varepsilon}}(0,x)-\smallU_d(0,x)\right|^2\right]\nu_d(dx)\\
&\leq 
\left(c_d\delta_{d,\varepsilon}\boundFG d^p +  c_d\left(\frac{\sqrt{e}(1+2LT)}{\sqrt{N_{d,\varepsilon}}}\right)^{N_{d,\varepsilon}} \right)^2
\leq
\left(\frac{\varepsilon}{4}+\frac{\varepsilon}{2}\right)^2< \varepsilon^2.
\end{split}
\end{align}
This implies that for all $d\in \N$, $\varepsilon \in (0,1]$ there exists $\omega_{d,\varepsilon}\in \Omega$ such that
\begin{align}\begin{split}
\int_{\R^d}\left|U^0_{N_{d,\varepsilon},N_{d,\varepsilon},d,\delta_{d,\varepsilon}}(0,x,\omega_{d,\varepsilon})-\smallU_d(0,x)\right|^2\nu_d(dx)
&<\varepsilon^2.
\end{split}\label{n06}
\end{align}
Next, observe that \cref{b26} shows that for all $d\in \N$, $\varepsilon \in (0,1]$ 
there exists $\Psi_{d,\varepsilon} \in \setCalN$ such that for all $x\in \R^d$
it holds that
$\funcCalR(\Psi_{d,\varepsilon})\in C(\R^d,\R)$,
$(\funcCalR(\Psi_{d,\varepsilon}))(x)=U^0_{N_{d,\varepsilon},N_{d,\varepsilon},d,\delta_{d,\varepsilon}}(0,x,\omega_{d,\varepsilon})$, $\supnorm{\funcCalL(\Psi_{d,\varepsilon} )}\leq k_{d,\delta_{d,\varepsilon}}(3 N_{d,\varepsilon})^{N_{d,\varepsilon}}$, and
\begin{align}\label{n09}
\cardL{\Psi_{d,\varepsilon}}= N_{d,\varepsilon}\left(\cardL{\netapproxf{\delta_{d,\varepsilon}}}-1\right)+\cardL{\netapproxg{d}{\delta_{d,\varepsilon}}}.
\end{align}
%and
%\begin{align}
%\supnorm{\funcCalL(\Psi_{d,\varepsilon} )}\leq k_{d,\delta_{d,\varepsilon}}(3 N_{d,\varepsilon})^{N_{d,\varepsilon}}.
%\end{align}
This and \cref{n06} prove \cref{n15}. Moreover,
\cref{n09} and \eqref{n02} imply that for all $d\in \N$, $\varepsilon \in (0,1]$ it holds that
\begin{align}\begin{split}\label{n21d}
\funcCalP(\Psi_{d,\varepsilon})&\leq \sum_{j=1}^{\cardL{\Psi_{d,\varepsilon}}}k_{d,\delta_{d,\varepsilon}}(3 N_{d,\varepsilon})^{N_{d,\varepsilon}} \left(k_{d,\delta_{d,\varepsilon}}(3 N_{d,\varepsilon})^{N_{d,\varepsilon}}+1\right)\\
&\leq 2\cardL{\Psi_{d,\varepsilon}}k_{d,\delta_{d,\varepsilon}}^2(3 N_{d,\varepsilon})^{2N_{d,\varepsilon}}\\
&=2
\left(  N_{d,\varepsilon}\left(\big.\!\cardL{\netapproxf{\delta_{d,\varepsilon}}}-1\right)+\cardL{\netapproxg{d}{\delta_{d,\varepsilon}}}\right)k_{d,\delta_{d,\varepsilon}}^2(3 N_{d,\varepsilon})^{2N_{d,\varepsilon}}.\end{split}
\end{align}
In addition, it follows from \cref{n21c} and \cref{n21e} that for all $\varepsilon \in (0,1]$ it holds that
\begin{equation}
\supnorm{\funcCalL(\netapproxf{\varepsilon})}\leq16\left[\max\!\left\{1,\left|L\big(4L+2|f(0)|\big)\right|^\frac{1}{(q-1)}\right\}\right]\varepsilon ^{-\frac{q}{(q-1)}}\le B \varepsilon^{-2}\le 
B \varepsilon^{-\alpha}.
\end{equation}
Combining this with \cref{n21d}, the fact that $\forall \, \varepsilon \in (0,1] \colon\cardL{\netapproxf{\varepsilon}}=3$, the fact that $\forall \, d\in \N, \varepsilon \in (0,1] \colon \supnorm{\funcCalL(\netapproxg{d}{\varepsilon})}\le d^p\varepsilon^{-\alpha}\boundFG$, and the fact that
$\forall \, d\in \N, \varepsilon \in (0,1] \colon \cardL{\netapproxg{d}{\varepsilon}}\le d^p\varepsilon^{-\beta}\boundFG$
implies that for all $d\in \N$, $\varepsilon \in (0,1]$ it holds that $k_{d,\delta_{d,\varepsilon}}\le d^p\delta_{d,\varepsilon}^{-\alpha}\boundFG$ and that
\begin{align}\begin{split}\label{n11}
\funcCalP(\Psi_{d,\varepsilon})
&\leq 2
\left(  N_{d,\varepsilon}\left(\big.\!\cardL{\netapproxf{\delta_{d,\varepsilon}}}-1\right)+\cardL{\netapproxg{d}{\delta_{d,\varepsilon}}}\right)
\left(d^{p}\delta_{d,\varepsilon}^{-\alpha}\boundFG\right)^2 (3 N_{d,\varepsilon})^{2N_{d,\varepsilon}}\\
&\leq 2
\left(  2N_{d,\varepsilon}+d^p(\delta_{d,\varepsilon})^{-\beta}\boundFG\right)
\left(d^{p}\delta_{d,\varepsilon}^{-\alpha}\boundFG\right)^2 (3 N_{d,\varepsilon})^{2N_{d,\varepsilon}}\\
&\leq 4d^p\delta_{d,\varepsilon}^{-\beta}\boundFG
d^{2p}\delta_{d,\varepsilon}^{-2\alpha}\boundFG^2 N_{d,\varepsilon} (3 N_{d,\varepsilon})^{2N_{d,\varepsilon}}\\
&= 4\boundFG^3 
(4c_d \boundFG d^p)^{2\alpha+\beta}
d^{3p} 
\varepsilon^{-(2\alpha+\beta)}
N_{d,\varepsilon} (3 N_{d,\varepsilon})^{2N_{d,\varepsilon}}.
\end{split}
\end{align}
Furthermore, \cref{n12} ensures that for all $d\in \N$, $\varepsilon \in (0,1]$ it holds that
\begin{equation}
\varepsilon \le 2c_d
\left(\frac{\sqrt{e}(1+2LT)}{\sqrt{N_{d,\varepsilon}-1}} \right)^{N_{d,\varepsilon}-1}.
\end{equation}
This together with \cref{n11} implies that for all $d\in \N$, $\varepsilon \in (0,1]$, $\gamma\in (0,1]$ it holds that
\begin{align}\begin{split}\label{n13}
&\funcCalP(\Psi_{d,\varepsilon})\leq
4\boundFG^{2\alpha+\beta+3} 
(4c_d)^{2\alpha+\beta}
d^{(2\alpha+\beta+3)p}
\varepsilon^{-(2\alpha+\beta)}
N_{d,\varepsilon} (3 N_{d,\varepsilon})^{2N_{d,\varepsilon}}\varepsilon^{4+\gamma}\varepsilon^{-(4+\gamma)}\\
&\leq
4\boundFG^{2\alpha+\beta+3}
(4c_d)^{4+2\alpha+\beta+\gamma}
d^{(2\alpha+\beta+3)p} 
N_{d,\varepsilon} (3 N_{d,\varepsilon})^{2N_{d,\varepsilon}}
\left(\frac{\sqrt{e}(1+2LT)}{\sqrt{N_{d,\varepsilon}-1}} \right)^{(N_{d,\varepsilon}-1)(4+\gamma)}\varepsilon^{-(4+2\alpha+\beta+\gamma)}\\
&\leq
4\boundFG^{2\alpha+\beta+3}
(4c_d)^{5+2\alpha+\beta}
d^{(2\alpha+\beta+3)p}
\sup_{n\in \N \cap [2,\infty)}\left[
n (3 n)^{2n}
\left(\frac{\sqrt{e}(1+2LT)}{\sqrt{n-1}} \right)^{(n-1)(4+\gamma)}
\right]
\varepsilon^{-(4+2\alpha+\beta+\gamma)}
\\
&=
4\boundFG^{2\alpha+\beta+3}
(4c_d)^{5+2\alpha+\beta}
d^{(2\alpha+\beta+3)p} 
\tilde C_\gamma
\varepsilon^{-(4+2\alpha+\beta+\gamma)}
.\end{split}
\end{align}
Combining this with \cref{n17} and \cref{n16} establishes that
there exist $\eta \in (0,\infty)$, $C=(C_\gamma)_{\gamma\in (0,1]}\colon (0,1]\to (0,\infty)$ such that for all
$d\in \N$, $\varepsilon, \gamma \in (0,1]$ it holds that
$\funcCalP(\Psi_{d,\varepsilon})\le C_\gamma d^\eta \varepsilon^{-(4+2\alpha+\beta+\gamma)}$. The proof of \cref{n18} is thus completed.
\end{proof}
\fussy

\subsection{Deep neural network approximations with general polynomial convergence rates}\label{subsec:dnn_approx_gen_pol}
\begin{corollary}\label{cor:main_thm}
Let $\left\|\cdot \right\|\colon (\cup_{d\in \N} \R^d) \to [0,\infty)$
and $\funcBoldA{d}\colon \R^d\to\R^d$, $d\in \N$, satisfy for all $d\in \N$, $x=(x_1,\ldots,x_d)\in \R^d$ that $\|x\|=[\sum_{i=1}^d(x_i)^2]^{1/2}$
and
$
\funcBoldA{d}(x)= \left(\max\{x_1,0\},\ldots,\max\{x_d,0\}\right),
$
let
$
\setCalN= \cup_{H\in  \N}\cup_{(k_0,k_1,\ldots,k_{H+1})\in \N^{H+2}}
[ \prod_{n=1}^{H+1} \left(\R^{k_{n}\times k_{n-1}} \times\R^{k_{n}}\right)],
$
let
$\funcCalR\colon \setCalN\to (\cup_{k,l\in \N} C(\R^k,\R^l))$ and
$\funcCalP\colon\setCalN\to \N$
 satisfy
for all $H\in \N$, $k_0,k_1,\ldots,k_H,k_{H+1}\in \N$,
$
\Phi = ((W_1,B_1),\ldots,(W_{H+1},B_{H+1}))\in \prod_{n=1}^{H+1} \left(\R^{k_n\times k_{n-1}} \times\R^{k_n}\right), 
$
$x_0 \in \R^{k_0},\ldots,x_{H}\in \R^{k_{H}}$ with 
$\forall\, n\in \N\cap [1,H]\colon x_n = \funcBoldA{k_n}(W_n x_{n-1}+B_n )$ 
that
\vspace{-5mm}
\begin{equation*}
\funcCalR(\Phi )\in C(\R^{k_0},\R ^ {k_{H+1}}),\quad
 (\funcCalR(\Phi)) (x_0) = W_{H+1}x_{H}+B_{H+1}, \quad\text{and}\quad \funcCalP(\Phi )=\textstyle{\sum\limits_{n=1}^{H+1}}k_n(k_{n-1}+1),
\end{equation*}
\vspace{-6mm}

\noindent
let
$T,\kappa \in (0,\infty)$, $f\in C(\R,\R)$, $(\netapproxg{d}{\varepsilon})_{d\in \N,\varepsilon\in (0,1]}\subseteq \setCalN$,
$(c_d)_{d\in\N}\subseteq (0,\infty)$,
for every $d\in \N$ let $g_d\in C(\R^d,\R)$, 
  for every $d\in \N$ let $u_d\in C^{1,2}([0,T]\times\R^d,\R)$,
%for every $\varepsilon\in (0,1]$, $d\in \N$ let $\netapproxg{d}{\varepsilon}\in \setCalN$,
and assume for all $d\in \N$, $v,w\in \R$, $x\in \R^d$, $\varepsilon \in (0,1]$, $t\in(0,T)$ that
$|f(v)-f(w)|\le \kappa|v-w|$,
$\funcCalR(\netapproxg{d}{\varepsilon})\in C(\R^d,\R)$,
$|(\funcCalR(\netapproxg{d}{\varepsilon}))(x)|\le \kappa d^\kappa (1+\normRd{x}^\kappa)$,
 $\left| g_d(x)-(\funcCalR(\netapproxg{d}{\varepsilon}))(x)\right| \le \varepsilon \kappa d^\kappa (1+\normRd{x}^\kappa)$,
 $\funcCalP(\netapproxg{d}{\varepsilon})\le \kappa d^\kappa \varepsilon^{-\kappa}$,
%$\left(\int_{\R^d}\normRd{y}^{2pq} \nu_d(dy)\right)^{\nicefrac{1}{(2pq)}}\le \boundFG d^{\mathfrak p}$,
 $|u_d(t,x)|\le c_d(1+\normRd{x}^{c_d})$,
$u_d(T,x)=g_d(x)$,
and
\begin{equation}  \begin{split}\label{eq:cor.PDE}
 (\tfrac{\partial}{\partial t}u_d)(t,x)+\tfrac{1}{2}(\Delta_xu_d)(t,x)+f(u_d(t,x))=0.
\end{split}     \end{equation}
Then
there exist $(\Psi_{d,\varepsilon})_{d\in \N,\varepsilon \in (0,1]}\subseteq \setCalN$, $\eta \in (0,\infty)$
%, $C=(C_\gamma)_{\gamma \in (0,1]}\colon (0,1]\to (0,\infty)$ 
such that for all
$d\in \N$, $\varepsilon\in (0,1]$ it holds that
$\funcCalR(\Psi_{d,\varepsilon})\in C(\R^d,\R)$, 
$\funcCalP(\Psi_{d,\varepsilon})\le\eta  d^\eta \varepsilon^{-\eta}$, and
\begin{equation}
\left[\int_{[0,1]^d}\left|\smallU_d(0,x)-(\funcCalR(\Psi_{d,\varepsilon}))(x)\right|^2\, dx\right]^{\!\nicefrac{1}{2}}\le \varepsilon.
\end{equation}
\end{corollary}

\begin{proof}[Proof of \cref{cor:main_thm}]
Throughout the proof assume without loss of generality that $\kappa \ge 2$,
let $\supnorm{\cdot} \colon (\cup_{d\in \N} \R^d) \to [0,\infty)$ and $\cardna \colon (\cup_{d\in \N}\R^d) \to \N$ satisfy for all $d\in \N$, $x=(x_1,\ldots,x_d)\in \R^d$ that $\supnorm{x}=\max_{i\in [1,d]\cap \N}|x_i|$ and $\card{x}=d$, let
$\funcCalL\colon \setCalN\to\setCalD$
%\begin{align}
%\funcCalL\colon \setCalN\to\setCalD\quad\text{and}\quad
%\funcCalR\colon \setCalN\to \bigcup_{k,l\in \N} C(\R^k,\R^l)\label{m07b}
%\end{align}
satisfy
for all $H\in \N$, $k_0,k_1,\ldots,k_H,k_{H+1}\in \N$,
$
\Phi = ((W_1,B_1),\ldots,(W_{H+1},B_{H+1}))\in \prod_{n=1}^{H+1} \left(\R^{k_n\times k_{n-1}} \times\R^{k_n}\right)
$
that
\begin{equation}
\funcCalL(\Phi)= (k_0,k_1,\ldots,k_{H}, k_{H+1}),
\end{equation}
and let $\boundFG =\max\left\{1, \kappa \left[\sup_{r\in [0,\infty)}\frac{1+r^\kappa}{(1+r)^\kappa}\right]\right\}$.
The fact that $\forall \, d\in \N$, $t\in [0,T]$, $x\in \R^d \colon |u_d(t,x)|\le c_d(1+\normRd{x}^{c_d})$, the fact that 
$\forall \, d\in \N$, $x\in \R^d \colon u_d(T,x)=g_d(x)$,
the fact that $\forall \, v,w\in \R\colon |f(v)-f(w)|\le \kappa|v-w|$, \cref{eq:cor.PDE}, and the Feynman-Kac-formula ensure that the functions $u_d$, $d\in \N$, satisfy \cref{n03}. 
Next note that for all $d\in \N$, $\varepsilon \in (0,1]$, $x\in \R^d$ it holds that
\begin{equation}\label{y01}
|(\funcCalR(\netapproxg{d}{\varepsilon}))(x)|\le \kappa d^\kappa (1+\normRd{x}^\kappa)
\le  \kappa \left[\sup_{r\in [0,\infty)}\frac{1+r^\kappa}{(1+r)^\kappa}\right]  d^\kappa  (1+\normRd{x})^\kappa
\le  \boundFG  d^\kappa  (1+\normRd{x})^\kappa,
\end{equation}
\begin{equation}\label{y02}
\left| g_d(x)-(\funcCalR(\netapproxg{d}{\varepsilon}))(x)\right| \le \varepsilon \kappa d^\kappa (1+\normRd{x}^\kappa)\le \varepsilon \kappa \left[\sup_{r\in [0,\infty)}\frac{1+r^\kappa}{(1+r)^\kappa}\right] d^\kappa (1+\normRd{x})^\kappa
\le 
\varepsilon \boundFG d^\kappa (1+\normRd{x})^{2\kappa},
\end{equation}
\begin{equation}\label{y03}
\supnorm{\funcCalL(\netapproxg{d}{\varepsilon})}\le \funcCalP(\netapproxg{d}{\varepsilon})\le \kappa d^\kappa \varepsilon^{-\kappa}\le \boundFG d^\kappa \varepsilon^{-\kappa},
\end{equation}
and
\begin{equation}\label{y04}
\cardL{\netapproxg{d}{\varepsilon}}\le \funcCalP(\netapproxg{d}{\varepsilon})\le \kappa d^\kappa \varepsilon^{-\kappa}
\le  \boundFG d^\kappa \varepsilon^{-\kappa}.
\end{equation}
Moreover, observe that the fact that $\forall\, d\in \N, y\in [0,1]^d\colon \normRd{y}\le \sqrt{d}$ ensures that 
for all $d\in \N$, $\alpha \in (0,\infty)$ it holds that
\begin{equation}\label{y05}
\left(\int_{[0,1]^d}\normRd{y}^\alpha \, dy\right)^{\nicefrac{1}{\alpha}} \le \sqrt{d} \le \boundFG d.
\end{equation}
Combining this with \cref{y01}--\cref{y04} and \cref{n18} (with $\alpha=\kappa$, $\beta=\kappa$, $\boundFG=\boundFG$, $\LipConstF=\kappa$, $p=\kappa$, $q=2$, $\mathfrak{p}=1$, and $\gamma=\frac{1}{2}$ in the notation of \cref{n18}) ensures that there exist $(\Psi_{d,\varepsilon})_{d\in \N,\varepsilon \in (0,1]}\subseteq \setCalN$, $\eta \in (0,\infty)$
such that for all
$d\in \N$, $\varepsilon\in (0,1]$ it holds that
$\funcCalR(\Psi_{d,\varepsilon})\in C(\R^d,\R)$, 
$\funcCalP(\Psi_{d,\varepsilon})\le\eta  d^\eta \varepsilon^{-\eta}$, and
\begin{equation}
\left[\int_{[0,1]^d}\left|\smallU_d(0,x)-(\funcCalR(\Psi_{d,\varepsilon}))(x)\right|^2\, dx\right]^{\!\nicefrac{1}{2}}\le \varepsilon.
\end{equation} 
The proof of \cref{cor:main_thm} is thus completed.
\end{proof}

\subsection*{Acknowledgements}
This project has been partially supported through the German Research Foundation via research
grants HU1889/6-1 (MH) and UR313/1-1 (TK) and via RTG 2131 \emph{High-dimensional Phenomena in Probability - Fluctuations and Discontinuity} (TAN).

%\bibliography{lit,bibfile}
%\bibliographystyle{plain}

\def\cprime{$'$}

\end{document}